%% file: Degree_rational_maps_and_specialization.tex
\documentclass[a4paper, 11pt]{amsart}   
\usepackage{amssymb,amscd,latexsym}  
\usepackage{amsmath}
\usepackage{amsthm}
\usepackage{mathdots}
\usepackage[pagebackref,colorlinks=true,linkcolor=blue,urlcolor=blue]{hyperref}
\usepackage{color}
\usepackage{tabularx}
\usepackage{amsfonts}
\usepackage{paralist}
\usepackage{aliascnt}
\usepackage[initials]{amsrefs}
\usepackage{amscd}
\usepackage{blkarray}
\usepackage{mathbbol}
\usepackage{setspace}
\usepackage[inner=2.5cm,outer=2.5cm, bottom=3.3cm]{geometry}
\usepackage{tikz, tikz-cd}

\usetikzlibrary{matrix}
\usetikzlibrary{arrows}

\newcommand{\kk}{\mathbb{k}}
\newcommand{\KK}{\mathbb{K}}

\newcommand{\GG}{\mathcal{G}}
\newcommand{\bgg}{\mathbb{g}}
\newcommand{\bGG}{\mathbb{G}}
\newcommand{\bmm}{\mathbb{M}}
\newcommand{\NN}{\normalfont\mathbb{N}}
\newcommand{\ZZ}{\normalfont\mathbb{Z}}

\newcommand{\MM}{\normalfont\mathcal{M}}
\newcommand{\PP}{\normalfont\mathbb{P}}
\newcommand{\xx}{\normalfont\mathbf{x}}
\newcommand{\yy}{\normalfont\mathbf{y}}

\newcommand{\mm}{{\normalfont\mathfrak{m}}}

\newcommand{\pp}{{\normalfont\mathfrak{p}}}

\newcommand{\qqq}{\mathfrak{q}}
\newcommand{\nn}{{\normalfont\mathfrak{N}}}
\newcommand{\nnn}{{\normalfont\mathfrak{n}}}

\newcommand{\rank}{\normalfont\text{rank}}

\newcommand{\depth}{\normalfont\text{depth}}
\newcommand{\grade}{\normalfont\text{grade}}
\newcommand{\Tor}{\normalfont\text{Tor}}

\newcommand{\Quot}{\normalfont\text{Quot}}

\newcommand{\HT}{\normalfont\text{ht}}

\newcommand{\Ann}{\normalfont\text{Ann}}
\newcommand{\Supp}{\normalfont\text{Supp}}

\newcommand{\Rees}{\mathcal{R}}
\newcommand{\Hom}{\normalfont\text{Hom}}
\newcommand{\Fitt}{\normalfont\text{Fitt}}

\newcommand{\OO}{\mathcal{O}}

\newcommand{\I}{\mathcal{I}}

\newcommand{\FF}{\normalfont\mathcal{F}}
\newcommand{\HL}{\normalfont\text{H}_{\mm}}
\newcommand{\HH}{\normalfont\text{H}}

\newcommand{\gr}{\normalfont\text{gr}}

\newcommand{\AAA}{\mathfrak{A}}

\newcommand{\bideg}{\normalfont\text{bideg}}
\newcommand{\Proj}{\normalfont\text{Proj}}
\newcommand{\Spec}{\normalfont\text{Spec}}
\newcommand{\MaxSpec}{\normalfont\text{MaxSpec}}
\newcommand{\multProj}{\normalfont\text{MultiProj}}
\newcommand{\biProj}{{\normalfont\text{BiProj}}}
\newcommand{\Rat}{\mathfrak{R}}

\newcommand{\fib}{\mathfrak{F}}
\newcommand{\sfib}{\widetilde{\mathfrak{F}_R(I)}}

\newtheorem{theorem}{Theorem}[section]

\newaliascnt{headcor}{headthm}

\aliascntresetthe{headcor}

\newaliascnt{headconj}{headthm}

\aliascntresetthe{headconj}

\newaliascnt{corollary}{theorem}
\newtheorem{corollary}[corollary]{Corollary}
\aliascntresetthe{corollary}

\newaliascnt{lemma}{theorem}
\newtheorem{lemma}[lemma]{Lemma}
\aliascntresetthe{lemma}

\newaliascnt{conjecture}{theorem}

\aliascntresetthe{conjecture}

\newaliascnt{proposition}{theorem}
\newtheorem{proposition}[proposition]{Proposition}
\aliascntresetthe{proposition}

\theoremstyle{definition}
\newaliascnt{definition}{theorem}
\newtheorem{definition}[definition]{Definition}
\aliascntresetthe{definition}

\newaliascnt{notation}{theorem}
\newtheorem{notation}[notation]{Notation}
\aliascntresetthe{notation}

\newaliascnt{example}{theorem}
\newtheorem{example}[example]{Example}
\aliascntresetthe{example}

\newaliascnt{examples}{theorem}

\aliascntresetthe{examples}

\newaliascnt{remark}{theorem}
\newtheorem{remark}[remark]{Remark}
\aliascntresetthe{remark}

\newaliascnt{problem}{theorem}

\aliascntresetthe{problem}

\newaliascnt{construction}{theorem}

\aliascntresetthe{construction}

\newaliascnt{setup}{theorem}
\newtheorem{setup}[setup]{Setup}
\aliascntresetthe{setup}

\newaliascnt{algorithm}{theorem}

\aliascntresetthe{algorithm}

\newaliascnt{observation}{theorem}

\aliascntresetthe{observation}

\newaliascnt{defprop}{theorem}
\newtheorem{defprop}[defprop]{Definition-Proposition}
\aliascntresetthe{defprop}

\def\equationautorefname~#1\null{(#1)\null}
\def\sectionautorefname~#1\null{Section #1\null}
\def\subsectionautorefname~#1\null{\S #1\null}
\def\trdeg{{\rm trdeg}}
\def\surjects{\twoheadrightarrow}

\begin{document}

\title{Degree of rational maps and specialization}

\author{Yairon Cid-Ruiz}
\address[Cid-Ruiz]{Max Planck Institute for Mathematics in the Sciences, Inselstra\ss e 22, 04103 Leipzig, Germany.}
\email{cidruiz@mis.mpg.de}
\urladdr{https://ycid.github.io}

\author{Aron Simis}

\address[Simis]{Dipartimento di Scienze Matematiche del Politecnico di Torino,
	C.so Duca degli Abruzzi 24, 
	10129 Torino, Italy.}
\email{aron.simis@polito.it}
\address[{Simis [current address]}]{Departamento de Matem\'atica, CCEN, 
	Universidade Federal de Pernambuco, 
	50740-560 Recife, PE, Brazil.}
\email{aron@dmat.ufpe.br}

\begin{abstract}
One considers the behavior of the degree of a rational map under specialization of the coefficients of the defining linear system. 
The method rests on the classical idea of Kronecker as applied to the context of projective schemes and their specializations. 
For the theory to work one is led to develop the details of rational maps and their graphs when the ground ring of coefficients is a Noetherian domain.
\end{abstract}



\maketitle



\section{Introduction}

The overall goal of this paper is to obtain bounds for the degree of a rational map in terms of the main features of its base ideal (i.e., the ideal generated by a linear system defining the map). 
In order that this objective stays within a reasonable limitation, one focuses on rational maps whose source and target are projective varieties. 
Although there is some recent progress in the multi-projective environment (see \cite{EFFECTIVE_BIGRAD} and \cite{MULTPROJ}), it is the present authors' believe that a thorough examination of the projective case is a definite priority.

Now, to become more precise one should rather talk about projective {\em schemes} as source and target of the envisaged rational maps. 
The  commonly sought interest is the case of projective schemes over a field (typically, but not necessarily, algebraically closed). 
After all, this is what core classical projective geometry is all about. 
Alas, even this classical setup makes it hard to look at the degree of a rational map since one has no solid grip on any general theory that commutative algebra lends, other than the rough skeleton of field extension degree theory.

One tactic that has often worked is to go  all the way up to a generic case and then find sufficient conditions for the specialization to keep some of the main features of the former. The procedure depends on  taking a dramatic number of variables to allow modifying the given data into a generic shape.
The method is seemingly due to Kronecker and was quite successful in the hands of Hurwitz (\cite{Hurwitz}) in establishing a new elegant theory of elimination and resultants.

\smallskip

Of a more recent crop, one has, e.g., \cite{Residual_int}, \cite{Generic_residual_int}, \cite{Ulrich_RedNo},  \cite{ram1}.

\smallskip

In a related way, one has the notion of when an ideal specializes modulo a regular sequence:  given an ideal $I\subset R$ in a ring, one says that  $I$ specializes with respect to a sequence of elements $\{a_1,\ldots, a_n\}\subset R$ if the latter is a regular sequence both on $R$ and on $R/I$.
A tall question in this regard is to find conditions under which the defining ideal of some well-known rings -- such as the Rees ring or the associated graded ring of an ideal  (see, e.g., \cite{EISENBUD_HUNEKE_SPECIALIZATION}, \cite{SIMIS_ULRICH_SPECIALIZATION}) -- specialize with respect to a given sequence of elements.
Often, at best one can only describe some obstructions to this sort of procedure, normally in terms of the kernel of the specialization map.

\smallskip

The core of the paper can be said to lie in between these two variations of specialization as applied to the situation of rational maps between projective schemes and their related ideal-theoretic objects. 

It so happens that at the level of the generic situation the coefficients live in a polynomial ring $A$ over a field, not anymore in a field.
This entails the need to consider rational maps defined by linear systems over the ring $A$, that is, rational maps with source $\PP_A^r$ or, a bit more generally,  an integral closed subscheme of $\PP_A^r$.
Much to our surprise, a complete such theory, with all the required details that include the ideal-theoretic transcription, is not easily available.
For this reason, the first part of the paper deals with such details with an eye for the  ideal-theoretic behavior concealed in or related to the geometric facts. 
Experts may feel that the transcription to a ground ring situation is the expected one, as indicated in several sources and, by and large, this feeling is correct. However, making sure that everything works well and give full rigorous proofs is a different matter altogether.

A tall order in these considerations will be a so-called {\em relative fiber cone} that mimics the notion of a fiber cone in the classical environment over a field -- this terminology is slightly misleading as the notion is introduced in algebraic language, associated to the concept of a Rees algebra rather than to the geometric version (blowup); however, one will draw on both the algebraic and the geometric versions. 
As in the case of a ground field, here too this fiber cone plays the role of the (relative) cone over the image of the rational map in question, and this fact will be called upon in many  theorems.

Another concept dealt with is the {\em saturated fiber cone}, an object perhaps better understood in terms of global sections of a suitable sheaf of rings.
The concept has been introduced in \cite{MULTPROJ} in the coefficient field environment and is presently extended to the case when the coefficients belong to a Noetherian domain  of finite Krull dimension.
It contains the relative fiber cone as a subalgebra and plays a role in rational maps, mainly as an obstruction to birationality in terms of this containment.
Also, its multiplicity is equal to the product of the degree of a rational map and the degree of the corresponding image.

With the introduction of these considerations, one will be equipped to tackle the problem of specializing rational maps, which is the main objective of this paper.
An important and (probably) expected application is that under a suitable general specialization of the coefficients, one shows that the degrees of the rational map and the corresponding image remain constant.

\smallskip

Next is a summary of the contents in each section.

Throughout the ground coefficient ring is a Noetherian domain $A$ of finite Krull dimension.

\smallskip

In \autoref{sec:notations}, one fixes the basic algebraic terminology and notation. This is a short section which may be skipped by those familiar with the current concepts concerning the Rees algebra (blowup algebra) and its associates.

\smallskip

In \autoref{sec:rat_maps}, the basics of rational maps between irreducible projective varieties over $A$ are developed. 
Particular emphasis is set towards a description of geometric concepts in terms of the algebraic analogs. 
For instance, and as expected, the image and the graph of a rational map are described in terms of the fiber cone and the Rees algebra, respectively. 
In the last part of this section, the notion of a saturated fiber cone is introduced and studied in the relative environment over $A$.

\smallskip

\autoref{sec:alg_tools} is devoted to a few additional algebraic tools of a more specific nature. With a view towards the specialization of determinantal-like schemes one reviews a few properties of generic determinantal rings. 
The main tool in this section goes in a different direction, focusing on upper bounds for the dimension of certain graded parts of  local cohomology modules of a finitely generated module over a bigraded algebra.

\smallskip

The core of the paper is  \autoref{sec:specialization}.
Here one assumes that the ground ring is a polynomial ring $A:=\kk[z_1,\ldots,z_m]$  over a field $\kk$ and specializes these variables to elements of $\kk$. Thus, one considers  a maximal ideal of the form $\nnn:=\left(z_1-\alpha_1,\ldots,z_m-\alpha_m\right)$.  
Since clearly $\kk \simeq A/\nnn$, the $A$-module structure of $\kk$ is given via the homomorphism $A \twoheadrightarrow A/\nnn \simeq \kk$.
One takes a standard graded polynomial ring $R:=A[x_0,\ldots,x_r]$ ($[R]_0=A$) and a tuple of forms $\{g_0,\ldots,g_s\} \subset R$ of the same positive degree. Let $\{\overline{g_0},\ldots,\overline{g_s}\} \subset R/\nnn R$ denote the corresponding tuple of forms in $R/\nnn R\simeq \kk[x_0,\ldots,x_r]$ where $\overline{g_i}$ is the image of $g_i$ under the natural homomorphism $R \twoheadrightarrow R/\nnn R$.

Consider the rational maps 
$$
\GG: \PP_A^r \dashrightarrow \PP_A^s \quad  \text{ and } \quad \bgg: \PP_\kk^r \dashrightarrow \PP_\kk^s
$$ 
determined by the tuples of forms $\{g_0,\ldots,g_s\}$ and $\{\overline{g_0},\ldots,\overline{g_s}\}$, respectively.
 
The main target is finding conditions under which the degree $\deg(\bgg)$ of $\bgg$ can be bounded above or below by the degree $\deg(\GG)$ of $\GG$.
The main result in this line is \autoref{THM_REDUCTION_REES}. 
In addition, set $\I:=(g_0,\ldots,g_s)\subset R$  and $I:=(\overline{g_0},\ldots,\overline{g_s})\subset R/\nnn R$.
Let $\mathbb{E}(\I)$ be the exceptional divisor of the blow-up of $\PP_A^r$ along $\I$.
A bit surprisingly, having a grip on the  dimension of the scheme $\mathbb{E}(\I) \times_{A} \kk$ is the main condition to determine whether $\deg(\bgg) \le \deg(\GG)$ or $\deg(\bgg) \ge \deg(\GG)$.
In order to control $\dim\left(\mathbb{E}(\I) \times_A \kk\right)$ one can impose some constraints on the analytic spread of $\I$ localized at certain primes, mimicking an idea in \cite{EISENBUD_HUNEKE_SPECIALIZATION} and \cite{SIMIS_ULRICH_SPECIALIZATION}. 
Also, under suitable general conditions, one shows that it is always the case that $\deg(\bgg) = \deg(\GG)$.

An additional interest in this section is the specialization of the saturated fiber cone  and the fiber cone of $\I$.
By letting the specialization be suitably general, it is proved in \autoref{specializing_mul_fiber}  and \autoref{cor:specialization_image} that the multiplicities of the saturated fiber cone and the fiber cone of $I$ are equal to the ones of the saturated fiber cone and the fiber cone of $\I \otimes_{A} \Quot(A)$, respectively, where $\Quot(A)$ denotes the field of fractions of $A$.

\smallskip 

\autoref{sec:applications} gives some applications to certain privileged objects.
Firstly, one considers the usual duet: equigenerated codimension $2$ perfect ideals and equigenerated codimension $3$ Gorenstein ideals over a standard graded polynomial ring over a field. 
Here, the focus is on the relation between the degree of the rational map and the standard degrees of a minimal set of generating syzygies, a subject much in vogue these days (see, e.g.,~\cite{AB_INITIO, Simis_cremona,KPU_blowup_fibers,EISENBUD_ULRICH_ROW_IDEALS,Hassanzadeh_Simis_Cremona_Sat,SIMIS_RUSSO_BIRAT,EFFECTIVE_BIGRAD,SIMIS_PAN_JONQUIERES,HASSANZADEH_SIMIS_DEGREES, HULEK_KATZ_SCHREYER_SYZ, MULTPROJ}).
While this relation is sufficiently available in the respective generic cases, not so much in a typical concrete family of such schemes.
This is where the main results on specialization come in to obtain a sharp relation between the two sorts of degrees mentioned.

At the end of the section one shows that the so-called $j$-multiplicity -- also very much in vogue (see e.g.~\cite{JMULT_MONOMIAL, JEFFRIES_MONTANO_VARBARO, COMPUTING_J_MULT, POLINI_XIE_J_MULT}) -- is stable under specialization.

\smallskip

For the more impatient reader, here is a pointer to the main results:  \autoref{thm:dim_cohom_bigrad},  \autoref{prop:minimal_primes_tensor_Rees},  \autoref{prop:find_open_set}, \autoref{THM_REDUCTION_REES},  \autoref{specializing_mul_fiber}, \autoref{cor:specialization_image}, \autoref{specializing_degree_HB}, \autoref{thm:Gor_ht_3} and \autoref{cor:j_mult}.

\section{Terminology and notation}
\label{sec:notations}

Let $R$ be a Noetherian ring and $I \subset R$ be an ideal.

\begin{definition}\label{Fitting_conditions}\rm
Let $m\geq 0$ be an integer (one allows $m=\infty$).
\begin{enumerate}
	\item [(G)] $I$ satisfies the {\em condition} $G_m$ if 
	$\mu(I_\mathfrak{p}) \le \HT(\pp)$ for all $\pp \in V(I) \subset \Spec(R)$ such that $\HT(\pp) \le m-1$.
	 \item [(F)] In addition, suppose that $I$ has a regular element.
	$I$ satisfies the {\em condition} $F_m$ if 
	$\mu(I_\mathfrak{p}) \le \HT(\pp) +1 - m$ for all $\pp\in \Spec(R)$ such that $I_{\pp}$ is not principal. Provided $I$ is further assumed to be principal locally in codimension at most $m-1$, the condition is equivalent to requiring that $\mu(I_\mathfrak{p}) \le \HT(\pp) +1 - m$ for all $\pp\in \Spec(R)$ containing $I$ such that $\HT(\pp)\geq m$.
\end{enumerate}
\end{definition}
In terms of Fitting ideals, $I$ satisfies $G_m$ if and only if $\HT(\Fitt_i(I)) > i$ for all $i<m$, whereas $I$ satisfies $F_m$ if and only if $\HT(\Fitt_{i}(I)) \geq m+i$ for all $i\geq 1$.
These conditions were originally introduced in \cite[Section 2, Definition]{AN} and \cite[Lemma 8.2, Remark 8.3]{HSV_TRENTO_SCHOOL}, respectively.
Both conditions are more interesting when the cardinality of a global set of generators of $I$ is large and $m$ stays low. Thus, $F_m$ is typically considered for $m=0,1$, while $G_m$ gets its way when $m\leq \dim R$.

\begin{definition}\rm
	The {\em Rees algebra}  of $I$ is defined as the $R$-subalgebra 
	$$
	\Rees_R(I):=R[It] = \bigoplus_{n\ge 0} I^nt^n  \subset R[t],
	$$
	and the {\em associated graded ring}  of $I$ is given by 
	$$
	\gr_I(R):=\Rees_R(I)/I \Rees_R(I) \simeq \bigoplus_{n\ge 0} I^n/I^{n+1}.
	$$
	If, moreover, $R$ is local, with maximal ideal $\mm$, one defines the {\em fiber cone} of $I$ to be 
		$$
	\fib_R(I): = \Rees_R(I)/\mm\Rees_R(I) \simeq \gr_I(R)/\mm\gr_I(R),
	$$
	and the {\em analytic spread} of $I$, denoted by $\ell(I)$, to be the (Krull) dimension of $\fib_R(I)$.
\end{definition}

The following notation will prevail throughout most of the paper.

\begin{notation}\rm
	\label{NOTA:prelim_Spec_A}
	Let $A$ be a Noetherian ring of finite Krull dimension.
	Let $(R,\mm)$ denote a standard graded algebra over a $A=[R]_0$ and $\mm$ be its graded irrelevant ideal $\mm=([R]_1)$. 
	Let $S := A[y_0,\ldots,y_s]$ denote a standard graded polynomial ring over $A$.	
\end{notation}

Let $I \subset R$ be a homogeneous ideal generated by $s+1$ polynomials $\{f_0,\ldots,f_s\} \subset R$ of the same degree $d>0$ -- in particular, $I=\left([I]_d\right)$.
Consider the bigraded $A$-algebra 
$$
\AAA:=R \otimes_A S=R[y_0,\ldots,y_s],
$$ 
where $\bideg([R]_1)=(1,0)$ and $\bideg(y_i)=(0,1)$.
By setting $\bideg(t)=(-d, 1)$, then $\Rees_R(I)=R[It]$ inherits a bigraded structure over $A$.
One has a bihomogeneous (of degree zero) $R$-homomorphism
\begin{equation}
\label{eq:present_Rees_alg}
\AAA \longrightarrow  \Rees_R(I) \subset R[t] ,
\quad y_i  \mapsto  f_it.
\end{equation} 
Thus, the bigraded structure of $\Rees_R(I)$ is given by 
\begin{equation}
	\label{eq_bigrad_Rees_alg}
	\Rees_R(I) = \bigoplus_{c, n \in \ZZ} {\left[\Rees_R(I)\right]}_{c, n} \quad \text{ and } \quad {\left[\Rees_R(I)\right]}_{c, n} = {\left[I^n\right]}_{c + nd}t^n.
\end{equation}
One is primarily interested in the $R$-grading of the Rees algebra, namely,
 ${\left[\Rees_R(I)\right]}_{c}=\bigoplus_{n=0}^\infty {\left[\Rees_R(I)\right]}_{c,n}$, and of particular interest is  
$$
{\left[\Rees_R(I)\right]}_{0}=\bigoplus_{n=0}^\infty {\left[I^n\right]}_{nd}t^n = A\left[[I]_dt\right] \simeq A\left[[I]_d\right]=\bigoplus_{n=0}^\infty {\left[I^n\right]}_{nd}\subset R.
$$
Clearly, $\Rees_R(I)=[\Rees_R(I)]_0\oplus \left(\bigoplus_{c\geq 1} [\Rees_R(I)]_c\right)=[\Rees_R(I)]_0\oplus \mm \Rees_R(I)$. Therefore, one gets
\begin{equation}
\label{eq_isom_Rees_0_special_fiber}
A\left[[I]_d\right]\simeq {\left[\Rees_R(I)\right]}_0\simeq \Rees_R(I)/\mm \Rees_R(I)
\end{equation}
as graded $A$-algebras.

\begin{definition}\label{relative_fiber_cone}\rm
Because of its resemblance to the fiber cone in the case of a local ring, one here refers to the right-most algebra in \autoref{eq_isom_Rees_0_special_fiber} as the (relative) {\em fiber cone} of $I$,
and often identify it with the $A$-subalgebra $A\left[[I]_d\right]\subset R$ by the above natural isomorphism.
It will also be denoted by $\fib_R(I)$.
\end{definition}

\begin{remark}\rm
	If $R$ has a distinguished or special maximal ideal $\mm$ (that is, if $R$ is graded with graded irrelevant ideal $\mm$ or if $R$ is local with maximal ideal $\mm$), then the fiber cone also receives the name of \textit{special fiber ring}.
\end{remark}

\section{Rational maps over a domain}
\label{sec:rat_maps}

In this part one develops the main points of the theory of rational maps with source and target projective varieties defined over an arbitrary Noetherian domain of finite Krull dimension.
Some of these results will take place in the case the source is a biprojective (more generally, a multi-projective) variety, perhaps with some extra work in the sleeve. 
From now on assume that $R$ is a domain, which in particular implies that $A={[R]}_0$ is also a domain. 
Some of the subsequent results will also work assuming that $R$ is reduced, but additional technology would be required. 

\subsection{Dimension}

In this subsection one considers a simple way of constructing chains of relevant graded prime ideals and draw upon it to algebraically describe the dimension of projective schemes. 
These results are possibly well-known, but one includes them anyway for the sake of completeness.

The following easy fact seems to be sufficiently known. 

\begin{lemma}
	\label{lem:extend_min_primes}
	Let $B$ be a commutative ring and $A \subset B$ a subring. 
	Then, for any minimal prime $\pp \in \Spec(A)$ there exists a minimal prime $\mathfrak{P} \in \Spec(B)$ such that $\pp = \mathfrak{P} \cap A$.
	\begin{proof}
		First, there is some prime of $B$ lying over $\pp$. Indeed, any prime ideal of the ring of fractions $B_{\pp}=B\otimes_A A_{\pp}$ is the image of a prime ideal $P\subset B$ not meeting $A\setminus \pp$, hence contracting to $\pp$.
		
		For any descending chain of prime ideals $P=P_0\supsetneq P_1\supsetneq \cdots$ such that $P_i\cap A \subseteq \pp$ for every $i$, their intersection $Q$ is prime and obviously $Q\cap A\subseteq \pp$.
		Since $\pp$ is minimal, then $Q \cap A = \pp$.
		
		Therefore, Zorn's lemma yields the existence of a minimal prime in $B$ contracting to $\pp$.
	\end{proof}
\end{lemma}

\begin{proposition}
	\label{prop:chain_graded_primes}
	Let $A$ be a Noetherian domain of finite Krull dimension $k=\dim(A)$ and let $R$ denote a finitely generated graded domain over $A$ with $[R]_0=A$.
	Let $\mm:=(R_+)$ be the graded irrelevant ideal of $R$. 
	If $\HT(\mm)\ge 1$, then there exists a chain of graded prime ideals
	$$
	0 = \mathfrak{P}_0 \subsetneq \cdots \subsetneq \mathfrak{P}_{k-1} \subsetneq \mathfrak{P}_k   
	$$
	such that $\mathfrak{P}_k \not\supseteq \mm$.
	\begin{proof}
		Proceed by induction on $k=\dim(A)$.
		
		The case $k=0$ it clear or vacuous. 
		Thus, assume that $k>0$.
		
	Let $\nnn$ be a maximal ideal of $A$ with $\HT(\nnn)=k$.
	By \cite[Theorem 13.6]{MATSUMURA} one can choose $0\neq a \in \nnn \subset A$ such that $\HT(\nnn/aA)=\HT(\nnn)-1$.
		Let $\qqq$ be a minimal prime of $aA$ such that $\HT(\nnn/\qqq)=\HT(\nnn)-1$.
		From the ring inclusion $A/aA \hookrightarrow R/aR$ (because $A/aA$ is injected as a graded summand) and \autoref{lem:extend_min_primes}, there is a  minimal prime $\mathfrak{Q}$ of $aR$ such that $\qqq = \mathfrak{Q} \cap A$.
		
		Clearly, $\mm\not\subseteq \mathfrak{Q}$. Indeed, otherwise $(\qqq,\mm)\subseteq \mathfrak{Q}$ and since $\mm$ is a prime ideal of $R$ of height at least $1$ then $(\qqq,\mm)$ has height at least $2$; this contradicts Krull's Principal Ideal Theorem since $\mathfrak{Q}$ is a minimal prime of a principal ideal. 
		
		Let $R^\prime = R/\mathfrak{Q}$ and $A^\prime=A/\qqq$. Then $R^\prime$ is a finitely generated graded algebra over $A^\prime$ with $[R^\prime]_0=A^\prime$ and $\mm':=([R^\prime]_+)=\mm R'$.
		Since $\mathfrak{Q} \not\supseteq \mm$, it follows $\HT(\mm R^\prime) \ge 1$ and by construction,  $\dim(A^\prime)=\dim(A)-1$.
		So by the inductive hypothesis there is a chain of graded primes $0 = \mathfrak{P}_0^\prime \subsetneq \cdots \subsetneq \mathfrak{P}_{k-1}^\prime$ in $R^\prime$ such that $\mathfrak{P}_{k-1}^\prime \not\supseteq \mm R^\prime$.
		Finally, for $j\ge 1$  define $\mathfrak{P}_j$ as the inverse image of $\mathfrak{P}_{j-1}^\prime$ via the surjection $R \twoheadrightarrow R^\prime$.		
	\end{proof}
\end{proposition}

Recall that $X:=\Proj(R)$ is a closed subscheme of $\PP_A^r$, for suitable $r$ ($=$ relative embedding dimension of $X$) whose underlying topological space is the set of all homogeneous prime ideals of $R$ not containing $\mm$ and it has a basis given by the open sets of the form $D_+(f):=\{\wp\in X| f\notin \wp\}$, where $f\in R_+$ is a homogeneous element of positive degree.
Here, the sheaf structure is given by the degree zero part of the homogeneous localizations 
$$
\Gamma\left(D_{+}(f), {\OO_{X}\mid}_{D_{+}(f)}\right): = R_{(f)} = \Big\{ \frac{g}{f^k} \mid g,f\in R, \deg(g) = k \deg(f) \Big\}.
$$

Let $K(X):= R_{(0)}$ denote the field of rational functions of $X$, where 
$$R_{(0)} = \Big\{ \frac{f}{g} \mid f,g \in R, \deg(f) =\deg(g), g\neq 0 \Big\},$$	
the degree zero part of the homogeneous localization of $R$ at the null ideal.

Likewise, denote $\PP_A^s=\Proj(S)=\Proj(A[y_0,\ldots,y_s])$.

The dimension $\dim (X)$ of the closed subscheme $X$ is defined to be the supremum of the lenghts of chains of irreducible closed subsets (see, e.g., \cite[Definition, p. 5 and p. 86]{HARTSHORNE}). 
The next result is possibly part of the dimensional folklore (cf. \cite[Lemma 1.2]{HYRY_MULTIGRAD}).

For any domain $D$, let $\Quot(D)$ denote its field of fractions.

\begin{corollary}\label{dimension_general}
	\label{lem_dim_X}
	For the integral subscheme $X=\Proj(R)\subset \PP_A^r$ one has
	$$
	\dim(X)=\dim(R)-1=\dim(A)+\trdeg_{\Quot(A)}\left(K(X)\right).
	$$
	\begin{proof}
		For any prime $\mathfrak{P} \in X$, the ideal $(\mathfrak{P},\mm)\neq R$ is an ideal properly containing $\mathfrak{P}$, hence the latter is not a maximal ideal.
		Therefore $\HT(\mathfrak{P}) \le \dim(R)-1$ for any $\mathfrak{P} \in X$, which clearly implies that $\dim(X) \le \dim(R)-1$.
		
		From \cite[Lemma 1.1.2]{simis1988krull} one gets the equalities 
		\begin{equation*}
		\dim\left(R\right)=\dim(A)+\HT(\mm)=\dim(A)+\trdeg_{\Quot(A)}(\Quot(R)).
		\end{equation*}
		There exists a chain of graded prime ideals $0 = \mathfrak{P}_0 \subsetneq \cdots \subsetneq \mathfrak{P}_{h-1} \subsetneq \mathfrak{P}_{h} = \mm$ such that $h=\HT(\mm)$ (see, e.g., \cite[Theorem 13.7]{MATSUMURA},  \cite[Theorem 1.5.8]{BRUNS_HERZOG}).
		Let $T=R/\mathfrak{P}_{h-1}$. 
		Since $\HT(\mm T)=1$, \autoref{prop:chain_graded_primes} yields the existence of a chain of graded prime ideals $0 = \mathfrak{Q}_0 \subsetneq \cdots \subsetneq \mathfrak{Q}_k$ in $T$, where $k=\dim(A)$ and $\mathfrak{Q}_k \not\supseteq \mm T$.
		By taking inverse images along the surjection $R \twoheadrightarrow T$, one obtains a chain of graded prime ideals not containing $\mm$ of length $h-1+k=\dim(R)-1$.
		Thus, one has the reverse inequality $\dim(X)\ge \dim(R)-1$.

		Now, for any $f \in {[R]}_1$, one has
		$
		\Quot(R) = R_{(0)}(f)
		$
		with  $f$ transcendental over $K(X)=R_{(0)}$.
		Therefore 
		$$
		\dim(X)=\dim(A)+\trdeg_{\Quot(A)}\left(\Quot(R)\right)-1=\dim(A)+\trdeg_{\Quot(A)}\left(K(X)\right),
		$$	
		as was to be shown.
	\end{proof}
\end{corollary}

Next, one deals with the general multi-graded case, which follows by using an embedding and reducing the problem to the single-graded setting.

Let $T=\bigoplus_{\mathbf{n} \in \NN^m} {[T]}_{\mathbf{n}}$ be a standard $m$-graded ring over ${[T]}_{\mathbf{0}}=A$, where $\mathbf{0}=(0,0,\ldots,0) \in \NN^m$.
The multi-graded irrelevant ideal in this case is given by $\nn = \bigoplus_{n_1>0,\ldots,n_m>0} {[T]}_{n_1,\ldots,n_m}$.
Here, one also assumes that $T$ is a domain.

Similarly to the single-graded case, one defines a multi-projective scheme from $T$.
The multi-projective scheme  $\multProj(T)$ is given by 
the set of all multi-homogeneous prime ideals in $T$ which do not contain $\nn$, and its scheme structure is obtained by using multi-homogeneous localizations. 
The multi-projective scheme $Z:=\multProj(T)$ is a closed subscheme of $\PP_A^{r_1} \times_{A} \PP_A^{r_2} \cdots \times_A \PP_A^{r_m}$, for suitable integers $r_1,\ldots,r_m$.

Let $T^{(\Delta)}$ be the single-graded ring $T^{(\Delta)}=\bigoplus_{n\ge 0} {[T]}_{(n,\ldots,n)}$, then the natural inclusion $T^{(\Delta)} \hookrightarrow T$ induces an isomorphism of schemes 
$Z=\multProj(T) \xrightarrow{\simeq} \Proj\left(T^{(\Delta)}\right)$ (see, e.g.,~\cite[Exercise II.5.11]{HARTSHORNE}), corresponding to the Segre embedding
$$
\PP^{r_1}_A \times_A \PP^{r_2}_A \times_A \cdots \times_A \PP^{r_m}_A \longrightarrow \PP^{(r_1+1)(r_2+1)\cdots(r_m+1)-1}_A.
$$

Since one is assuming that $Z$ is an integral scheme, the field of rational functions of $Z$ is given by 
$$
K(Z):=T_{(0)} = \Big\{ \frac{f}{g} \mid f,g \in T, \deg(f) =\deg(g), g\neq 0 \Big\}.	
$$

The following result yields a multi-graded version of \autoref{dimension_general}.

\begin{corollary}\label{lem:dim_biProj}
	For the integral subscheme $Z=\multProj(T) \subset \PP_A^{r_1} \times_{A} \PP_A^{r_2} \cdots \times_A \PP_A^{r_m}$ one has 
	$$
	\dim(Z) = \dim(T)-m=\dim(A)+\trdeg_{\Quot(A)}\left(K(Z)\right).
	$$
	\begin{proof}
		From the isomorphism $Z \simeq \Proj\left(T^{(\Delta)}\right)$ and \autoref{dimension_general}, it follows that $$\dim(Z)=\dim(A)+\trdeg_{\Quot(A)}\left(K(Z)\right).
		$$
		The dimension formula of \cite[Lemma 1.1.2]{simis1988krull} gives 
		$
		\dim(T)=\dim(A) + \trdeg_{\Quot(A)}\left(\Quot(T)\right).
		$
		
		Choose homogeneous elements $f_1 \in {[T]}_{(1,0,\ldots,0)}, f_2 \in {[T]}_{(0,1,\ldots,0)}, \ldots, f_m \in {[T]}_{(0,0,\ldots,1)}$.
		Then, one has 
		$
		\Quot(T) = T_{(0)}(f_1,f_2,\ldots,f_m)
		$
		with $\{f_1,f_2,\ldots,f_m\}$ a transcendence basis over $K(Z)=T_{({0})}$. 
		Therefore
		$
			\dim(Z)=\dim(A) + \trdeg_{\Quot(A)}\left(\Quot(T)\right)-m
			=\dim(T)-m,
		$
		and so the result follows.
	\end{proof}
\end{corollary}

A generalization for closed subschemes of $\PP^{r_1}_A \times_A \PP^{r_2}_A \times_A \cdots \times_A \PP^{r_m}_A$ is immediate.

\begin{corollary}
	\label{cor:dim_mult_proj_sub_scheme}
	For a closed subscheme $W=\multProj(C) \subset \PP^{r_1}_A \times_A \PP^{r_2}_A \times_A \cdots \times_A \PP^{r_m}_A$ one has 
	$$
	\dim(W) = \max\big\lbrace \dim\left(C/\pp\right)-m \mid \pp \in W \cap {\normalfont\text{Min}}(C) \big\rbrace.
	$$
\end{corollary}

\subsection{Main definitions}

One restates the following known concept.
\begin{definition}
	\label{def:rat_map}
	\rm Let $\Rat(X, \PP_A^s)$ denote the set of pairs $(U,\varphi)$ where $U$ is a dense open subscheme of $X$ and where $\varphi:U \rightarrow \PP_A^s$ is a morphism of $A$-schemes. 
	Two pairs $(U_1,\varphi_1), (U_2,\varphi_2) \in \Rat(X, \PP_A^s)$ are said to be {\em equivalent} if there exists a dense open subscheme $W \subset U_1 \cap U_2$ such that ${\varphi_1\mid}_{W}={\varphi_2\mid}_{W}$.
	This gives an equivalence relation on $\Rat(X, \PP_A^s)$.
	A {\em rational map} is defined to be  an equivalence class in $\Rat(X, \PP_A^s)$ and any element of this equivalence class is said to define the rational map.
\end{definition}

A rational map as above is denoted $\FF:X \dashrightarrow \PP_A^s$, where the dotted arrow reminds one that typically it will not be defined everywhere as a map.
In \cite[Lecture 7]{Harris} (see also \cite{AB_INITIO}) it is explained that, in the case where $A$ is a field the above definition is equivalent to a more usual notion of a rational map in terms of homogeneous coordinate functions.
Next, one proceeds to show that the same is valid in the relative environment over $A$.

First it follows from the definition that any morphism $U \rightarrow \PP_A^s$ as above from a  dense open subset defines a unique rational map $X \dashrightarrow \PP_A^s$.
Now, let there be given $s+1$ forms $\mathbf{f}=\{f_0,f_1,\ldots,f_s\} \subset R$ of the same degree $d>0$.
Let $\mathfrak{h}:S\rightarrow R$ be the graded homomorphism of $A$-algebras given by 
\begin{align*}
\mathfrak{h}:S=A[y_0,y_1,\ldots,y_s] &\longrightarrow R\\
y_i & \mapsto f_i.
\end{align*}
There corresponds to it a morphism of $A$-schemes 
\begin{equation*}
\label{eq:Proj_map}
\Phi(\mathbf{f})=\Proj(\mathfrak{h}): D_+(\mathbf{f}) \longrightarrow \Proj(S) = \PP_A^s
\end{equation*}
where $D_+(\mathbf{f}) \subset \Proj(R) = X$ is the open subscheme given by 
$$
D_+(\mathbf{f}) = \bigcup_{i=0}^s D_+(f_i).
$$
Therefore, a set of $s+1$ forms $\mathbf{f}=\{f_0,f_1,\ldots,f_s\} \subset R$ of the same positive degree determines a unique rational map given by the equivalence class of $\left(D_+(\mathbf{f}), \Phi(\mathbf{f}) \right)$ in $\mathfrak{R}(X,\PP_A^s)$.

\begin{definition}\label{f-coordinate}\rm
	Call $\Phi(\mathbf{f})$ the {\em $\mathbf{f}$-coordinate morphism} and denote the corresponding rational map by $\FF_{\mathbf{f}}$.
\end{definition}

Conversely: 

\begin{lemma}
	\label{lem:find_representative}
	Any rational map $\FF:X=\Proj(R) \dashrightarrow \PP_A^s$ is of the form $\FF_{\mathbf{f}}$, where $\mathbf{f}$ are forms of the same positive degree.
	\begin{proof}
		Let $U$ be a dense open subset in $X$ and $\varphi:U \rightarrow \PP_A^s$ be a morphism, such that the equivalence class of the pair $(U, \varphi)$ in $\Rat(X, \PP_A^s)$ is equal to $\FF$.
		
		Consider $V=D_+(y_0)$ and $W = {\varphi}^{-1}(V)$ and 
		restrict to an affine open subset, $W^\prime =\Spec(R_{(\ell)}) \subset W$, where $\ell \in R$ is a homogeneous element of positive degree. 	
		It yields a morphism ${\varphi\mid}_{W^\prime}:W^\prime \rightarrow V$, that corresponds to a ring homomorphism
		$
		\tau : S_{(y_0)} \rightarrow R_{(\ell)}
		$.
		For each $0<i\le s$ one has
		$$
		\tau\left(\frac{y_i}{y_0}\right) = \frac{g_i}{\ell^{\alpha_i}}
		$$	
		where $\deg(g_i)=\alpha_i\deg(\ell)$.
		Setting $\alpha:=\max_{1\le i \le s} \{\alpha_i\}$, one writes
		$$
		f_0: = \ell^\alpha \quad \text{ and }\quad f_i: =\ell^\alpha\frac{g_i}{\ell^{\alpha_i}}= \ell^{\alpha-\alpha_i}g_i \text{ for } 1 \le i \le s.
		$$
		By construction, ${\varphi\mid}_{W^\prime}={\Phi(\mathbf{f})\mid}_{W^\prime}$, where $\Phi(\mathbf{f})$ denotes the $\mathbf{f}$-coordinate morphism determined by $\mathbf{f}=\{f_0,\ldots,f_s\}$, as in definition \autoref{f-coordinate}, hence $\FF=\FF_{\mathbf{f}}$ where $\mathbf{f}=\{f_0,\ldots,f_s\}$.
	\end{proof}
\end{lemma}

Given a rational map $\FF:X \dashrightarrow \PP_A^s$, any ordered $(s+1)$-tuple $\mathbf{f}=\{f_0,f_1,\ldots,f_s\}$ of forms of the same positive degree such that $\FF=\FF_{\mathbf{f}}$ is called a {\em representative} of the rational map $\FF$.

The following result explains the flexibility of representatives of the same rational map.

\begin{lemma}
	\label{lem_propor_representatives}
	Let $\mathbf{f}=\{f_0,\ldots,f_s\}$ and $\mathbf{f}'=\{f'_0,\ldots,f'_s\}$ stand for representatives of a rational map $\FF:X \dashrightarrow \PP_A^s$.
	Then $(f_0:\cdots :f_s)$ and $(f'_0:\cdots :f'_s)$ are proportional coordinate sets in the sense that there exist homogeneous forms $h,h'$ of positive degree such that $hf^\prime_i=h^\prime f_i$ for $i=0,\ldots,s$.
\end{lemma}
\begin{proof}
	Proceed similarly to \autoref{lem:find_representative}.
	Let $\Phi(\mathbf{f}):D_+(\mathbf{f}) \rightarrow \PP_A^s$ and $\Phi(\mathbf{f^\prime}):D_+(\mathbf{f^\prime}) \rightarrow \PP_A^s$ be morphisms as in \autoref{f-coordinate}.
	Let $V = \Spec\left(D_+(y_0)\right)$ and choose $W=\Spec(R_{(\ell)})$ such that $W \subset {\Phi(\mathbf{f})}^{-1}\left(V\right) \,\cap\, {\Phi(\mathbf{f^\prime})}^{-1}\left(V\right)$ and ${\Phi(\mathbf{f})\mid}_W={\Phi(\mathbf{f^\prime})\mid}_W$.
	
	The morphisms ${\Phi(\mathbf{f})\mid}_W:W \rightarrow V$ and ${\Phi(\mathbf{f^\prime})\mid}_W:W \rightarrow V$ correspond with the ring homomorphisms $\tau: S_{(y_0)}\rightarrow R_{(\ell)}$ and $\tau^\prime: S_{(y_0)}\rightarrow R_{(\ell)}$ such that
	$$
	\tau\left(\frac{y_i}{y_0}\right) = \frac{f_i}{f_0} \quad \text{ and }\quad \tau^\prime\left(\frac{y_i}{y_0}\right) = \frac{f^\prime_i}{f^\prime_0},
	$$ 
	respectively.
	Since this is now an affine setting, the ring homomorphisms $\tau$ and $\tau^\prime$ are the same (see e.g. ~\cite[Theorem 2.35]{GORTZ_WEDHORN}, ~\cite[Proposition II.2.3]{HARTSHORNE}).
	It follows that, for every $i=0,\ldots,s$, $f'_i/f'_0=f_i/f_0$ as elements of the homogeneous total ring of quotients of $R$.
	Therefore, there are homogeneous elements $h,h^\prime\in R$ ($h=f_0, h^\prime=f_0^\prime$) such that $hf^\prime_i=h^\prime f_i$ for $i=0,\ldots,s$. The claim now follows.
\end{proof}

In the above notation, one often denotes $\FF_{\mathbf{f}}$ simply by $(f_0:\cdots :f_s)$ and use this symbol for a representative of $\FF$.

\begin{remark}\label{identity_map}\rm
	Note that the identity morphism of $\PP^r_A$ is a rational map of $\PP^r_A$ to itself with natural representative $(x_0:\cdots:x_r)$ where $\PP^r_A=\Proj(A[x_0,\ldots,x_r])$. 
	Similarly, the identity morphism of $X=\Proj(R)$ is a rational map represented by $(x_0:\cdots:x_r)$, where now $x_0,\ldots,x_r$ generate the $A$-module $[R]_1$, and it is denoted by $\text{Id}_X$.
\end{remark}

The following sums up a version of \cite[Proposition 1.1]{Simis_cremona} over a domain. 
Due to \autoref{lem_propor_representatives}, the proof is a literal transcription of the proof in loc.~cit.

\begin{proposition}
	Let $\FF:X \dashrightarrow \PP_A^s$ be a rational map with representative $\mathbf{f}$. Set $I=(\mathbf{f})$.
	Then, the following statements hold:
	\begin{enumerate}[\rm (i)]
		\item The set of representatives of $\FF$ corresponds bijectively to the non-zero homogeneous vectors in the rank one graded $R$-module $\Hom_R(I, R)$. 
		\item If $\grade(I)\ge 2$, any representative of $\FF$ is a multiple of $\mathbf{f}$ by a homogeneous element in $R$.
	\end{enumerate}
\end{proposition}

\begin{remark}\rm
	If $R$ is in addition an UFD then any rational map has a unique representative up to a multiplier -- this is the case, e.g., when $A$ is a UFD and $R$ is a polynomial ring over $A$.	
\end{remark}

One more notational convention: if $\mathbf{f}=\{f_0,\ldots,f_s\}$ are forms of the same degree, $A[\mathbf{f}]$ will denote the $A$-subalgebra of $R$ generated by these forms.

An important immediate consequence is as follows:

\begin{corollary}
	\label{cor:isom_algebras_represent}
	Let $\mathbf{f}=(f_0:\cdots:f_s)$ and $\mathbf{f}'=(f'_0:\cdots:f'_s)$ stand for representatives of the same rational map $\FF:X=\Proj(R) \dashrightarrow \PP_A^s$. 
	Then $A[\mathbf{f}] \simeq A[\mathbf{f}']$ as graded $A$-algebras and $\Rees_R(I)\simeq \Rees_R(I^\prime)$ as bigraded $\AAA$-algebras, where $I=(\mathbf{f})$ and $I'=(\mathbf{f}')$.
\end{corollary}
\begin{proof}
	Let $\mathcal{J}$ and $\mathcal{J}^\prime$ respectively denote the ideals of defining equations of $\Rees_R(I)$ and $\Rees_R(I^\prime)$, as given in \autoref{eq:present_Rees_alg}.
	From \autoref{lem_propor_representatives}, there exist homogeneous elements $h,h^\prime \in R$ such that $hf_i^\prime=h^\prime f_i$ for $i=0,\ldots,s$.
	Clearly, then $I\simeq I^\prime$ have the same syzygies, hence the defining ideals $\mathcal{L}$ and $\mathcal{L}'$ of the respective symmetric algebras coincide. Since $R$ is a domain  and $I$ and $I'$ are nonzero, then
	$\mathcal{J}=\mathcal{L}:I^{\infty}=\mathcal{L}':I'^{\infty}=\mathcal{J}^\prime.$
	Therefore, $\Rees_R(I) \simeq \AAA/\mathcal{J} = \AAA/\mathcal{J}^\prime \simeq \Rees_R(I^\prime)$ as bigraded $\AAA$-algebras. 
	Consequently, 
	$$
	A[\mathbf{f}] \simeq \Rees_R(I)/\mm\Rees_R(I) \simeq \Rees_R(I^\prime)/\mm\Rees_R(I^\prime) \simeq A[\mathbf{f^\prime}]
	$$
	as graded $A$-algebras.
\end{proof}

\subsection{Image, degree and birational maps}

This part is essentially a recap on the algebraic description of the image, the degree and the birationality of a rational map in the relative case. Most of the material here has been considered in a way or another as a previsible extension of the base field situation (see, e.g., \cite[Theorem 2.1]{Laurent_Jouanolou_Closed_Image}).  

\begin{defprop}\label{image}\rm
	Let $\FF:X \dashrightarrow \PP_A^s$ be a rational map. The {\em image} of $\FF$ is equivalently defined as:
	\begin{enumerate}
		\item[{\rm (I1)}] The closure of the image of a morphism $U\rightarrow \PP_A^s$ defining $\FF$, for some (any)  dense open subset $U \subset X$.
		\item[{\rm (I2)}] The closure of the image of the $\mathbf{f}$-coordinate morphism $\Phi(\mathbf{f})$, for some (any) representative $\mathbf{f}$ of $\FF$.
		\item[{\rm (I3)}] $\Proj\left(A[\mathbf{f}]\right)$, for some (any) representative $\mathbf{f}$ of $\FF$, up to degree normalization of $A[\mathbf{f}]$.
	\end{enumerate} 
\end{defprop}
\begin{proof}
	The equivalence of (I1) and (I2) is clear by the previous developments.
	To check that (I2) and (I3) are equivalent, consider the ideal sheaf $\mathcal{J}$ given as the kernel of the natural homomorphism 
	\begin{equation*}
	\OO_{\PP_A^s}\rightarrow {\Phi(\mathbf{f})}_* \OO_{D_+(\mathbf{f})}.
	\end{equation*}
	It defines a closed subscheme $Y\subset \PP_A^s$ which corresponds to the schematic image of $\Phi(\mathbf{f})$ (see, e.g., \cite[Proposition 10.30]{GORTZ_WEDHORN}).
	The underlying topological space of $Y$ coincides with the closure of the image of $\Phi(\mathbf{f})$.
	Now, for any $0\le i\le s$,  ${\OO_{\PP_A^s}}\left({D_+(y_i)}\right)=S_{(y_i)}$ and ${\left({\Phi(\mathbf{f})}_* \OO_{D_+(\mathbf{f})}\right)}\left({D_+(y_i)}\right)=R_{(f_i)}$.
	Then, for $0\le i\le s$, there is an exact sequence 
	$$
	0 \rightarrow {\mathcal{J}}\left({D_+(y_i)}\right) \rightarrow S_{(y_i)} \rightarrow R_{(f_i)}.
	$$
	Thus, ${\mathcal{J}}\left({D_+(y_i)}\right)=J_{(y_i)}$ for any  $0\le i\le s$, where $J$ is the kernel of the $A$-algebra homomorphism $\alpha:S\rightarrow A[\mathbf{f}] \subset R$ given by $y_i\mapsto f_i$.
	This implies that $\mathcal{J}$ is the sheafification of $J$.  Therefore, $Y \simeq \Proj(S/J) \simeq \Proj(A[\mathbf{f}]).$
\end{proof}

Now one considers the degree of a rational map $\FF:X \dashrightarrow \PP_A^s$. By \autoref{image}, the field of rational functions of the image $Y$ of $\FF$ is 
$$
K(Y)= {A[\mathbf{f}]}_{(0)}, 
$$
where $\mathbf{f}=(f_0:\cdots:f_s)$ is a representative of $\FF$.
Here $A[\mathbf{f}]$ is naturally $A$-graded as an $A$-subalgebra of $R$, but one may also consider it as a standard $A$-graded algebra by a degree normalization.

One gets a natural field extension 	$K(Y)= {A[\mathbf{f}]}_{(0)} \hookrightarrow R_{(0)} = K(X).$

\begin{definition}\label{degree}\rm
	The {\em degree} of $\FF:X \dashrightarrow \PP_A^s$ is
	$$
	\deg(\FF):=\left[K(X):K(Y)\right].
	$$
\end{definition}
One says  that $\FF$ is {\em generically finite} if $\left[K(X):K(Y)\right]<\infty$. 
If the field extension $K(X)|K(Y)$ is infinite, one agrees to say that $\FF$ has no well-defined degree (also, in this case, one often says that $\deg(\FF)=0$).

The following properties are well-known over a coefficient field. Its restatement in the relative case is for the reader's convenience.

\begin{proposition}
	\label{prop:generically_finite}
	Let $\FF:X\dashrightarrow \PP_A^s$ be a rational map with image $Y\subset \PP_A^s$. 
	\begin{enumerate}[\rm (i)]
		\item Let $\mathbf{f}$ denote a representative of $\FF$ and let $\Phi(\mathbf{f})$ be the associated $\mathbf{f}$-coordinate morphism.
		Then, $\FF$ is generically finite if and only if there exists a dense open subset $U \subset Y$ such that $\Phi(\mathbf{f})^{-1}(U) \rightarrow U$ is a finite morphism.
		\item  $\FF$ is generically finite if and only if $\dim(X)=\dim(Y)$.
	\end{enumerate}
\end{proposition}
\begin{proof}
	(i)	Let $\Phi(\mathbf{f}):D_+(\mathbf{f}) \subset X \rightarrow Y \subset \PP_A^s$ be the $\mathbf{f}$-coordinate morphism of $\FF$.
	One has an equality of fields of rational functions $K(X)=K\left(D_+(\mathbf{f})\right)$.
	But on $D_+(\mathbf{f})$ the rational map $\FF$ is defined by a morphism, in which case the result is given in \cite[Exercise II.3.7]{HARTSHORNE}.
	
	(ii) 
	By \autoref{lem_dim_X} one has $\dim(X)=\dim(A)+\trdeg_{\Quot(A)}\left(K(X)\right)$ and		
	by the same token,  $\dim(Y)=\dim(A)+{\rm trdeg}_{\Quot(A)}\left(K(Y)\right)$.
	It follows that 
	$$\dim(X)=\dim(Y) \;\Leftrightarrow\; {\rm trdeg}_{\Quot(A)}\left(K(X)\right)={\rm trdeg}_{\Quot(A)}\left(K(Y)\right).$$
	Since the later condition is equivalent to $\trdeg_{K(Y)}(K(X))=0$, one is through.
\end{proof}

Next one defines birational maps in the relative environment over $A$.
While any of the three alternatives below sounds equally fit as a candidate (as a deja vu of the classical coefficient field setup), showing that they are in fact mutually equivalent requires a small bit of work.  

\begin{defprop}\label{def:birational_maps}\rm
	Let $\FF:X \subset \PP_A^r \dashrightarrow\PP_A^s$ be a rational map with image $Y \subset \PP_A^s$. 
	The map $\FF$ is said to be {\em birational onto its image} if one of the following equivalent conditions is satisfied:
	\begin{enumerate}
		\item[\rm(B1)]  $\deg(\FF)=1$, that is $K(X)=K(Y)$.
		\item[\rm(B2)] There exists some dense open subset $U\subset X$ and a morphism $\varphi:U \rightarrow \PP_A^s$ such that the pair $(U,\varphi)$ defines $\FF$ and such that $\varphi$ is an isomorphism onto a  dense open subset  $V \subset Y$.
		\item[\rm(B3)] 
		There exists a rational map $\GG:Y \subset \PP_A^s \dashrightarrow X \subset \PP_A^r$ such that, for some (any)  representative $\mathbf{f}$ of $\FF$ and some (any) representative $\mathbf{g}=(g_0:\cdots:g_r)$  of $\GG$, the composite 
		$$
		\mathbf{g}(\mathbf{f})=\left(g_0(\mathbf{f}):\cdots:g_r(\mathbf{f})\right)
		$$ 
		is a representative of the identity rational map on $X$.
	\end{enumerate}
	\begin{proof}
		$\text{(B1)} \Rightarrow \text{(B2)}$.
		Let $\varphi^\prime:U^\prime \rightarrow \PP_A^s$ be a morphism from a dense open subset $U'\subset X$ such that $(U',\varphi')$ defines $\FF$.
		Let $\eta$ denote the generic point of $X$ and $\xi$ that of $Y$. 
		The field inclusion $\OO_{Y,\xi} \simeq K(Y) \hookrightarrow K(X) \simeq \OO_{X,\eta}$ coincides with the induced local ring homomorphism  
		$$
		{\left(\varphi^\prime\right)}_\eta^\sharp \,:\, \OO_{Y,\xi} \rightarrow \OO_{X,\eta}.
		$$
		Since by assumption $\deg(\FF)=1$,  ${\left(\varphi^\prime\right)}_\eta^\sharp$ is an isomorphism. 
		Then, by \cite[Proposition 10.52]{GORTZ_WEDHORN}  ${(\varphi^\prime)}_\eta^\sharp$ ``extends'' to an isomorphism from an open neighborhood $U$  of $\eta$ in $X$  onto an open neighborhood $V$ of $\xi$ in $Y$.
		Now, take the restriction $\varphi={\varphi^\prime\mid}_{U}:U\xrightarrow{\simeq} V$ as the required isomorphism.
	
		$\text{(B2)} \Rightarrow \text{(B3)}$ 
		Let $\varphi:U\subset X \xrightarrow{\simeq} V \subset Y$ be a morphism defining $\FF$, which is an isomorphism from a dense open subset $U \subset X$ onto a dense open subset $V \subset Y$.
		Let $\psi=\varphi^{-1}:V\subset Y\xrightarrow{\simeq} U\subset X$ be the inverse of $\varphi$.
		Let $\GG:Y \subset \PP_A^s \dashrightarrow X \subset \PP_A^r$ be the rational map defined by $(V,\psi)$. 
		
		Let $\text{Id}_X$ be the identity rational map on $X$ (\autoref{identity_map}).
		Take any representatives $\mathbf{f}=(f_0:\cdots:f_s)$ of $\FF$ and $\mathbf{g}=(g_0:\cdots:g_r)$ of $\GG$.
		Let $\GG \circ \FF$ be the composition of $\FF$ and $\GG$, i.e. the rational map defined by $(U, \psi \circ \varphi)$.
		Since $\psi \circ \varphi$ is the identity morphism on $U$, \autoref{def:rat_map} implies that the pair $(U, \psi \circ \varphi)$ gives the equivalence class of $\text{Id}_X$.
		Thus, one has $\text{Id}_X=\GG \circ \FF$, and by construction $\mathbf{g}(\mathbf{f})$ is a representative of $\GG\circ\FF$.

		$\text{(B3)} \Rightarrow \text{(B1)}$
		This is quite clear: take a representative $(f_0:\cdots:f_s)$ of $\FF$ and let $\GG$ and $(g_0:\cdots:g_r)$ be as in the assumption.
		Since the identity map of $X$ is defined by the representative $(x_0:\cdots:x_r)$, where $[R]_1=Ax_0+\cdots +Ax_r$ (see \autoref{identity_map}), then \autoref{lem_propor_representatives} yields the existence of nonzero (homogeneous) $h,h'\in R$ such that $h\cdot g_{i}(\mathbf{f})=h'\cdot x_{i}$, for $i=0,\ldots,r$.
		Then, for suitable $e\geq 0$,
		$$\frac{x_{i}}{x_0}=\frac{g_{i}(\mathbf{f})}{g_0(\mathbf{f})}=\frac{f_0^e\left(g_{i}(f_1/f_0,\ldots,f_m/f_0)\right)}{f_0^e\left(g_{0}(f_1/f_0,\ldots,f_m/f_0)\right)}
		=\frac{g_{i}(f_1/f_0,\ldots,f_m/f_0)}{g_{0}(f_1/f_0,\ldots,f_m/f_0)},\quad i=0,\ldots,r.$$
		This shows the reverse inclusion $K(X)\subset K(Y)$. 
	\end{proof}
\end{defprop}

\subsection{The graph of a rational map}

The tensor product $\AAA:=R\otimes_A A[\mathbf{y}]\simeq R[\mathbf{y}]$ has a natural structure of a standard bigraded $A$-algebra. Accordingly, the fiber product $\Proj(R) \times_A \PP_A^s$ has a natural structure of a biprojective scheme over $\Spec (A)$.
Thus, $\Proj(R) \times_A \PP_A^s=\biProj (\AAA)$.

The graph of a rational map $\FF:X=\Proj(R) \dashrightarrow \PP_A^s$ is a subscheme of this structure, in the following way:

\begin{defprop}\label{graph}\rm
	The {\em graph} of $\FF$ is equivalently defined as:
	\begin{enumerate}
		\item[{\rm (G1)}] The closure of the image of the morphism $(\iota,\varphi): U\rightarrow  X \times_A \PP_A^s$, where $\iota: U \hookrightarrow X$ is the natural inclusion and $\varphi:U\rightarrow \PP_A^s$ is a morphism from some (any) dense open subset defining $\FF$.
		\item[{\rm (G2)}] For some (any) representative $\mathbf{f}$ of $\FF$, the closure of the image of the morphism 
		$ \left(\iota,\Phi(\mathbf{f})\right) : D_+(\mathbf{f}) \longrightarrow X \times_A \PP_A^s,
		$
		where $\iota:D_+(\mathbf{f}) \hookrightarrow X$ is the natural inclusion and $\Phi(\mathbf{f}): D_+(\mathbf{f}) \rightarrow \PP_A^s$ is the $\mathbf{f}$-coordinate morphism.
		\item[{\rm (G3)}] $\biProj\left(\Rees_R(I)\right)$, where $I=(\mathbf{f})$ for some (any) representative $\mathbf{f}$ of $\FF$.
	\end{enumerate}
\end{defprop}
\begin{proof}
	The equivalence of (G1) and (G2) is clear, so one proceeds to show that (G2) and (G3) give the same scheme.
	Recall that, as in \autoref{eq:present_Rees_alg}, the Rees algebra of an ideal such as $I$ is a bigraded $\AAA$-algebra.
	The proof follows the same steps of the argument for the equivalence of (I2) and (I3) in the definition of the image of $\FF$ (cf. ~\autoref{image}).
	
	Let $\Gamma(\mathbf{f})$ denote the morphism as in (G2) and let $\mathfrak{G}\subset X\times_A\PP_A^s$ denote its schematic image.
	The underlying topological space of $\mathfrak{G}$ coincides with the closure of the image of $\Gamma(\mathbf{f})$.
	Then the ideal sheaf of $\mathfrak{G}$ is the kernel $\mathfrak{J}$ of the 
	corresponding homomorphism of ring sheaves
	\begin{equation}
	\label{eq:map_with_sheafs_graph_map}
	\OO_{X\times_A\PP_A^s}\rightarrow {\Gamma(\mathbf{f})}_* \OO_{D_+(\mathbf{f})}.
	\end{equation}
	Since the irrelevant ideal of $\AAA$ is $([R]_1)\cap(\yy)$, by letting $[R]_1=Ax_0+\cdots+Ax_r$ one can see that an affine open cover is given by $\Spec\left(\AAA_{(x_iy_j)}\right)$ for $0\le i\le r$ and $0\le j \le s$, where $\AAA_{(x_iy_j)}$ denotes the degree zero part of the bihomogeneous localization at powers of $x_iy_j$, to wit 
	\begin{equation}
	\label{eq:bihomogeneous_localiz}
	\AAA_{(x_iy_j)}=\Big\{\frac{g}{(x_iy_j)^\alpha} \mid g \in \AAA \;\text{ and }\;  \bideg(g)=(\alpha, \alpha) \Big\}.
	\end{equation}
	One has ${\OO_{X\times_A\PP_A^s}}\left({D_+(x_iy_j)}\right)=\AAA_{(x_iy_j)}$ and ${\left({\Gamma(\mathbf{f})}_* \OO_{D_+(\mathbf{f})}\right)}\left({D_+(x_iy_j)}\right)=R_{(x_if_j)}$, for $0\le i\le r$ and $0\le j \le s$.
	Then \autoref{eq:map_with_sheafs_graph_map} yields the exact sequence 
	$$
	0 \rightarrow {\mathfrak{J}}\left({D_+(x_iy_j)}\right) \rightarrow \AAA_{(x_iy_j)} \rightarrow R_{(x_if_j)}.
	$$	
	Let $\mathcal{J}$ be the kernel of the homomorphism of bigraded $\AAA$-algebras $\AAA\rightarrow \Rees_R(I) \subset R[t]$ given by $y_i\mapsto f_it$.
	The fact that ${\Rees_R(I)}_{(x_if_jt)} \simeq R_{(x_if_j)}$, yields the equality ${\mathfrak{J}}\left({D_+(x_iy_j)}\right)=\mathcal{J}_{(x_iy_j)}$.
	It follows that $\mathfrak{J}$ is the sheafification of $\mathcal{J}$.
	Therefore,  $\mathfrak{G} \simeq \biProj(\AAA/\mathcal{J}) \simeq \biProj(\Rees_R(I))$.
\end{proof}

\subsection{Saturated fiber cones over a domain}\label{sectio_on_saturation}

In this part one introduces the notion of a saturated fiber cone over a domain, by closely lifting from the ideas in \cite{MULTPROJ}.
As will be seen, the notion is strongly related to the degree and birationality of rational maps. 

For simplicity, assume that $R =A[\xx]=A[x_0,\ldots,x_r]$, a standard graded polynomial ring over $A$ and set $\KK:=\Quot(A), \mm=(x_0,\ldots,x_r)$.

The central object is the following graded $A$-algebra
$$
\sfib: = \bigoplus_{n=0}^\infty {\left[\big(I^n:\mm^\infty\big)\right]}_{nd},
$$
which one calls the {\em saturated fiber cone} of $I$.

Note the natural inclusion of graded $A$-algebras $\fib_R(I) \subset \sfib $.

For any $i\geq 0$, the local cohomology module $\HL^i(\Rees_R(I))$ has a natural structure of bigraded $\Rees_R(I)$-module, which comes out of the fact that 
$\HL^i(\Rees_R(I)) = \HH_{\mm\Rees_R(I)}^i(\Rees_R(I))$ (see also \cite[Lemma 2.1]{DMOD}).
In particular,  each $R$-graded part 
$$
{\left[\HL^i(\Rees_R(I))\right]}_j
$$
has a natural structure of graded $\fib_R(I)$-module.

Let $\Proj_{R\text{-gr}}(\Rees_R(I))$ denote the Rees algebra $\Rees_R(I)$ viewed as a ``one-sided'' graded $R$-algebra.	

\begin{lemma}
	\label{lem:props_special_fib_ring}
	With the above notation, one has:
	\begin{enumerate}[\rm (i)]
		\item There is an isomorphism of graded $A$-algebras 
		$$
		\sfib \,\simeq \, \HH^0\left(\Proj_{R\text{-gr}}(\Rees_R(I)), \OO_{\Proj_{R\text{-gr}}\left(\Rees_R(I)\right)}\right).
		$$
		\item $\sfib$ is a finitely generated graded $\fib_R(I)$-module.
		\item There is an exact sequence 
		$$
		0 \rightarrow \fib_R(I) \rightarrow \sfib \rightarrow {\left[\HL^1(\Rees_R(I))\right]}_0 \rightarrow 0
		$$
		of finitely generated graded $\fib_R(I)$-modules.
		\item If $A \rightarrow A^\prime$ is a flat ring homomorphism, then there is an isomorphism of graded $A'$-algebras
		$$
		\sfib \otimes_A A^\prime\simeq \widetilde{\mathfrak{F}_{R'}(IR^\prime)},
		$$
		where $R'=R \otimes_A A^\prime$.
	\end{enumerate}
	\begin{proof}
		(i) Since $\Rees_R(I)\;\simeq\; \bigoplus_{n=0}^\infty {I^n}(nd)$, by computing \v{C}ech cohomology with respect to the affine open covering ${\left(\Spec\left({\Rees_R(I)}_{(x_i)}\right)\right)}_{0\le i \le r}$ of ${\Proj_{R\text{-gr}}(\Rees_R(I))}$, one obtains 
		\begin{align*}
		\HH^0\left({\Proj_{R\text{-gr}}(\Rees_R(I))}, \OO_{{\Proj_{R\text{-gr}}(\Rees_R(I))}}\right)	&\simeq \bigoplus_{n\ge 0} \HH^0\left(\Proj(R),{(I^n)}^\sim(nd) \right)\\
		&=\bigoplus_{n=0}^\infty {\left[\big(I^n:\mm^\infty\big)\right]}_{nd} \quad \text{(\cite[Exercise II.5.10]{HARTSHORNE}).}\\
		&= \sfib.
		\end{align*}
		
		(ii) and  (iii) From \cite[Corollary 1.5]{HYRY_MULTIGRAD} (see also \cite[Theorem A4.1]{EISEN_COMM}) and the fact that $\HL^0(\Rees_R(I))=0$, there is an exact sequence 
		$$
		0 \rightarrow {\left[\Rees_R(I)\right]}_0 \rightarrow \HH^0({\Proj_{R\text{-gr}}(\Rees_R(I))}, \OO_{{\Proj_{R\text{-gr}}(\Rees_R(I))}})\simeq\sfib \rightarrow  {\left[\HL^1(\Rees_R(I))\right]}_0 \rightarrow 0
		$$
		of  $\fib_R(I)$-modules. 
		Now ${\left[\HL^1(\Rees_R(I))\right]}_0$ is a finitely generated module over $\fib_R(I)$ (see, e.g., \cite[Proposition 2.7]{MULTPROJ}, \cite[Theorem 2.1]{CHARDIN_POWERS_IDEALS}), thus implying that $\sfib$ is also finitely generated over $\fib_R(I)$.
		
		(iv) Since  $A \rightarrow A'$ is flat, base change yields 
		$$
		\HH^0\left(B, \OO_{B}\right) \simeq \HH^0\left({\Proj_{R\text{-gr}}(\Rees_R(I))}, \OO_{{\Proj_{R\text{-gr}}(\Rees_R(I))}}\right) \otimes_A A^\prime,
		$$
		where $B={\Proj_{R\text{-gr}}(\Rees_R(I))}\times_{A} A^\prime=\Proj_{R\text{-gr}}\left(\Rees_R(I) \otimes_A A^\prime\right)$.
		Also $\Rees_R(I) \otimes_A A^\prime \simeq \Rees_{R^\prime}(IR^\prime)$, by flatness, hence the result follows.
	\end{proof}
\end{lemma}

Let $\FF:\PP_A^r \dashrightarrow \PP_A^s$ be a rational map with representative $\mathbf{f}=(f_0:\cdots:f_s)$.
Let $\GG:\PP_\KK^r\dashrightarrow \PP_\KK^s$ denote a rational map with representative $\mathbf{f}$, where each $f_i$ is considered as an element of $\KK[\xx]$.
	Set $I=(\mathbf{f})\subset R$.

\begin{remark}\rm
\label{rem:degree_rat_map_over_KK}
The rational map $\FF$ is generically finite if and only if the rational map $\GG$ is so, and one has the equality
$
\deg(\FF)=\deg\left(\GG\right).
$
In fact, let $Y$ and $Z$ be the images of $\FF$ and $\GG$, respectively. 
	Since $K(\PP_A^r)=R_{(0)}={\KK[\xx]}_{(0)}=K(\PP_\KK^r)$ and $K(Y)={A[\mathbf{f}]}_{(0)}={\KK[\mathbf{f}]}_{(0)}=K(Z)$, then the statement follows from \autoref{degree} and \autoref{prop:generically_finite}.
\end{remark} 
	
The following result is a simple consequence of \cite[Theorem 2.4]{MULTPROJ}.

\begin{theorem}
	Suppose that $\FF$ is generically finite. 
	Then, the following statements hold:
	\begin{enumerate}[\rm (i)]
		\item $\deg(\FF)=\left[\sfib: \fib_R(I)\right]$.
		\item $e\left(\sfib \otimes_A \KK\right)=\deg(\FF)\cdot e\Big(\fib_R(I) \otimes_A \KK\Big)$, where $e(-)$ stands for Hilbert-Samuel multiplicity.
		\item Under the additional condition of $\fib_R(I)$ being integrally closed, then $\FF$ is birational if and only if $\sfib=\fib_R(I)$.		
	\end{enumerate}
	\begin{proof}
		(i) Let $\GG:\PP_\KK^r \dashrightarrow \PP_\KK^s$ be the rational map as above.
		Since $A \hookrightarrow \KK$ is flat,  $\fib_R(I) \otimes_A \KK\cong \fib_{\KK[\xx]}(I\KK[\xx])$ and $\sfib \otimes_A \KK\cong \widetilde{\fib_{\KK[\xx]}(I\KK[\xx])}$ (\autoref{lem:props_special_fib_ring} (iv)).
		
		Thus from \cite[Theorem 2.4]{MULTPROJ}, one obtains $\deg(\GG)=\left[\sfib \otimes_A \KK: \fib_R(I) \otimes_A \KK\right]$.
		It is clear that $\left[\sfib \otimes_A \KK: \fib_R(I) \otimes_A \KK\right]=\left[\sfib: \fib_R(I)\right]$.
		Finally, \autoref{rem:degree_rat_map_over_KK} yields the equality $\deg(\FF)=\deg(\GG)$.
		
		(ii) It follows from the associative formula for multiplicity (see, e.g., \cite[Corollary 4.7.9]{BRUNS_HERZOG}).
		
		(iii) It suffices to show that, assuming that $\FF$ is birational onto its image and that $\fib_R(I)$ is integrally closed, then  $\sfib=\fib_R(I)$.
		 Since $\deg(\FF)=1$,   part (i) gives
		$$
		\Quot\left(\sfib\right) = \Quot\Big(\fib_R(I)\Big).
		$$	
		Since $\fib_R(I)$ is integrally closed and $\fib_R(I) \hookrightarrow \sfib$ is an integral extension (see \autoref{lem:props_special_fib_ring}(ii)), then $\sfib=\fib_R(I)$.
	\end{proof}
\end{theorem}

\section{Additional algebraic tools}
\label{sec:alg_tools}

In this section one gathers a few algebraic tools to be used in the specialization of rational maps. 
The section is divided in two subsections, and each subsection deals with a different theme that is important on its own.

\subsection{Grade of certain generic determinantal ideals}

One provides lower bounds for the grade of certain generic determinantal ideals with a view towards treating other such ideals by specialization (see \autoref{sec:applications}).

For this part, $R$ will exceptionally denote an arbitrary  Noetherian  ring.

The next lemma deals with a well-known generic deformation of an ideal in $R$ (see, e.g., \cite[Proposition 3.2]{ram1} for a similar setup).

\begin{lemma}
	\label{lem:ht_specialization}
	Let $n,m \ge 1$, $\mathbf{z}=\{z_{i,j}\} \, (1\le i \le n, 1\le j \le m)$ be indeterminates and  $S:=R[\mathbf{z}]$. 
	Given an $R$-ideal $I=(f_1,\ldots,f_m) \subset R$, introduce 
the $S$-ideal $J=(p_1,p_2,\ldots,p_n) \subset S$, where
	$$
	p_i=f_1z_{i,1}+	f_2z_{i,2}+\cdots +f_mz_{i,m}.
	$$
	Then  $\grade(J)\ge \min\{n,\grade(I)\}$.
	\begin{proof}
		Let $Q$ be a prime ideal containing $J$.
		If $Q$ contains all the $f_i$'s, then $\depth(S_Q)\ge\grade(I)$.
		Otherwise, say, $f_1 \not\in Q$.
		Then one can write
		$
		\frac{p_i}{f_1}=z_{i,1}+ \frac{f_2}{f_1}z_{i,2}+\cdots +\frac{f_m}{f_1}z_{i,m} \in R_{f_1}[\mathbf{z}]
		$
		as elements of the localization $S_{f_1}=R_{f_1}[\mathbf{z}]$.
		Since $\{z_{1,1},\ldots, z_{n,1}\}$ is a regular sequence in $R_{f_1}[\mathbf{z}]$, then so is the sequence $\{p_1/f_1,\ldots,p_n/f_1\}$. Then, clearly
		 $\{p_1,\ldots,p_n\}$ is a regular sequence in $R_{f_1}[\mathbf{z}]$, hence $\depth(S_Q) \ge \grade\left(JR_{f_1}[\mathbf{z}]\right)\ge n$.		
	\end{proof}
\end{lemma}

As an expected follow-up, one considers lower bounds for the grade of certain determinantal ideals. The proof of next proposition, provided  for the sake of completeness,  follows from the well-known routine of inverting/localizing at a suitable entry.

\begin{proposition}
	\label{prop:ht_gen_det_ideals}
	Let $r,s \ge 1$. For every $1\le j \le s$, 
	let $I_j=\left(f_{j,1},\ldots,f_{j,m_j}\right) \subset R$ be an $R$-ideal.
	Let $\mathbf{z}=\{z_{i,j,k}\}$ be indeterminates, with indices ranging over $1\le i \le r$, $1\le j \le s$ and $1 \le k \le m_j$.
 Setting $S:=R[\mathbf{z}]$, consider the $r\times s$ matrix 
	$$
	M = \left( \begin{array}{cccc}
	p_{1,1} & p_{1,2} & \cdots & p_{1,s} \\
	p_{2,1} & p_{2,2} & \cdots & p_{2,s}\\
	\vdots & \vdots & & \vdots\\
	p_{r,1} & p_{r,2} & \cdots & p_{r,s}\\
	\end{array}
	\right)$$	
	where 	$p_{i,j} = f_{j,1}z_{i,j,1}  + f_{j,2}z_{i,j,2}  + \cdots + f_{j,m_j}z_{i,j,m_j}\in S$, for all $i,j$.
	Then 
	$$
	\grade\big(I_t(M)\big) \ge \min\{r-t+1,g\},
	$$
	where $1 \le t \le \min\{r,s\}$ and $g=\min_{1\le j \le s}\{\grade(I_j)\}$.
	\begin{proof}
		Proceed by induction on $t$. 
		The case $t=1$ follows from \autoref{lem:ht_specialization} since $I_1(M)$ is generated by the $p_{i,j}$'s themselves.
		
		Now suppose that $1 < t \le \min\{r,s\}$.
		Let $Q$ be a prime ideal containing $I_t(M)$.
		If $Q$ contains all the polynomials $p_{i,j}$, then again \autoref{lem:ht_specialization} yields $\depth(S_{Q}) \ge \min\{r,g\} \ge \min\{r-t+1,g\}$.
		Otherwise, say, $p_{r,s}\not\in Q$.
		
		Let $M^\prime$ denote the $(r-1)\times(s-1)$ submatrix of $M$ of the first $r-1$ rows and first $s-1$ columns. 
		 Clearly, 
		$
		I_{t-1}\left(M^\prime\right)S_{p_{r,s}} \,\subset\, I_{t}\left(M\right)S_{p_{r,s}}
		$
		in the localization  $S_{p_{r,s}}$.
		The inductive hypothesis gives 
		\begin{align*}
		\depth(S_Q) &\ge 
		\grade\left(I_{t}\left(M\right)S_{p_{r,s}}\right)	\\
		&\ge \grade\left(I_{t-1}\left(M^\prime\right)S_{p_{r,s}}\right)\\
		&\ge 	\grade\left(I_{t-1}\left(M^\prime\right)\right) \\
		&\ge \min\{(r-1)-(t-1)+1,g\}.
		\end{align*}
		Therefore, $\depth(S_Q)\ge \min\{r-t+1,g\}$ as was to be shown.
	\end{proof}
\end{proposition}

\subsection{Local cohomology of bigraded algebras}

The goal of this subsection is to establish upper bounds for the dimension of certain graded parts of the local cohomology of a finitely generated module over a bigraded algebra.
Analogous bounds are also derived for the sheaf cohomologies of the corresponding sheafified module.
The bounds are far reaching generalizations of \cite[Proposition 3.1]{MULTPROJ}, which has been  useful under various situations (see, e.g., \cite[Proof of Theorem 3.3]{MULTPROJ}, \cite[Proof of Theorem A]{MULT_SAT_PERF_HT_2}). 

The following setup will prevail along this part only. 

\begin{setup}
	Let $\kk$ be a field.
	Let $\AAA$ be a standard bigraded $\kk$-algebra, i.e., $\AAA$ can be generated over $\kk$ by the elements of bidegree $(1,0)$ and $(0,1)$. 
	Let $R$ and $S$ be the standard graded $\kk$-algebras given by $R={\left[\AAA\right]}_{(*,0)}=\bigoplus_{j\ge 0} {\left[\AAA\right]}_{j,0}$ and $S={\left[\AAA\right]}_{(0,*)}=\bigoplus_{k\ge 0} {\left[\AAA\right]}_{0,k}$, respectively.
	One sets $\mm = R_+ = \bigoplus_{j>0} {\left[\AAA\right]}_{j,0}$ and $\nnn = S_+ = \bigoplus_{k>0} {\left[\AAA\right]}_{0,k}$.
\end{setup}

Let $\bmm$ be a bigraded module over $\AAA$.
Denote by ${\left[\bmm\right]}_j$ the ``one-sided'' $R$-graded part
\begin{equation}
	\label{eq_one_side_grading}
	{\left[\bmm\right]}_j = \bigoplus_{k\in \ZZ
	} {\left[\bmm\right]}_{j,k}.
\end{equation}
Then, ${\left[\bmm\right]}_{j}$ has a natural structure of graded $S$-module, and its $k$-th graded part is given by 
$
{\left[{\left[\bmm\right]}_j\right]}_k={\left[\bmm\right]}_{j,k}.
$
Note that, for any $i\geq 0$, the local cohomology module $\HL^i(\bmm)$ has a natural structure of bigraded $\AAA$-module, and this can be seen from the fact that 
$\HL^i(\bmm) = \HH_{\mm\AAA}^i(\bmm)$ (also, see \cite[Lemma 2.1]{DMOD}).
In particular,  each $R$-graded part 
$
{\left[\HL^i(\bmm)\right]}_j
$
has a natural structure of graded $S$-module.

By considering $\AAA$ as a ``one-sided'' graded $R$-algebra, one gets the projective scheme $\Proj_{R\text{-gr}}(\AAA)$.
The sheafification of $\bmm$, denoted by $\widetilde{\bmm}$, yields a quasi-coherent $\OO_{\Proj_{R\text{-gr}}(\AAA)}$-module.

For any finitely generated bigraded $\AAA$-module $\bmm$,  ${\left[\HL^i(\bmm)\right]}_j$ and $\HH^i\left(\Proj_{R\text{-gr}}(\AAA), \widetilde{\bmm}(j)\right)$ are finitely generated graded $S$-modules for any $i\ge 0,j\in \ZZ$ (see, e.g., \cite[Proposition 2.7]{MULTPROJ}, \cite[Theorem 2.1]{CHARDIN_POWERS_IDEALS}).

The next theorem contains the main result of this subsection.

\begin{theorem}
	\label{thm:dim_cohom_bigrad}
	Let $\bmm$ be a finitely generated bigraded $\AAA$-module. 
	Then, the following inequalities hold 
	\begin{enumerate}[\rm (i)]
		\item $
		\dim\left({\left[\HL^i(\bmm)\right]}_j\right) \le \min\{\dim(\bmm) - i, \dim(S) \}
		$,
		\item $\dim\left(\HH^i\left(\Proj_{R\text{-gr}}(\AAA), \widetilde{\bmm}(j)\right)\right) \le \min\{\dim(\bmm) - i-1, \dim(S) \}$
	\end{enumerate}
	for all $i\ge 0$ and $j\in \ZZ$.
	\begin{proof}
		Let $d=\dim(\bmm)$.
		Since ${\left[\HL^i(\bmm)\right]}_j$ and $\HH^i\left(\Proj_{R\text{-gr}}(\AAA), \widetilde{\bmm}(j)\right)$ are finitely generated $S$-modules, it is clear that 
		$\dim\left({\left[\HL^i(\bmm)\right]}_j\right) \le \dim(S)$ and $\dim\left(\HH^i\left(\Proj_{R\text{-gr}}(\AAA), \widetilde{\bmm}(j)\right)\right)\le \dim(S)$.
		
		(i)		
		By the well-known Grothendieck Vanishing Theorem (see, e.g., \cite[Theorem 6.1.2]{Brodmann_Sharp_local_cohom}), $\HL^i(\bmm)=0$  for $i > d$, so that one takes $i \le d$. 
				
		Proceed by induction on $d$.
		
		Suppose that $d=0$.
		Then ${[\bmm]}_{j,k}=0$ for $k\gg 0$. 
		Since ${\left[\HL^0(\bmm)\right]}_j \subset {[\bmm]}_j$, one has  ${\left[{\left[\HL^0(\bmm)\right]}_j\right]}_k={\left[\HL^0(\bmm)\right]}_{j,k}=0$ for $k \gg 0$.
		Thus, $\dim\left({\left[\HL^0(\bmm)\right]}_j\right)=0$. 
		
		Suppose that $d>0$.
		There exists a finite filtration 
		$
		0 = \bmm_0 \subset \bmm_1 \subset \cdots \subset \bmm_n=\bmm
		$
		of $\bmm$ such that $\bmm_l/\bmm_{l-1} \cong \left(\AAA/\pp_l\right)(a_l,b_l)$ where $\pp_l \subset \AAA$ is a bigraded prime ideal with dimension $\dim(\AAA/\pp_l)\le d$ and $a_l,b_l \in \ZZ$.
		The short exact sequences
		\begin{equation*}
			0 \rightarrow \bmm_{l-1} \rightarrow \bmm_l \rightarrow \left(\AAA/\pp_l\right)(a_l,b_l) \rightarrow 0		
		\end{equation*}
		induce the following long exact sequences in local cohomology 
		\begin{equation*}
			{\left[\HL^i\left(\bmm_{l-1}\right)\right]}_j \rightarrow {\left[\HL^i\left(\bmm_{l}\right)\right]}_j \rightarrow {\left[\HL^i\left(\left(\AAA/\pp_l\right)(a_l,b_l)\right)\right]}_j.  
		\end{equation*}
		By iterating on $l$, one gets
		\begin{equation*}
			\dim\left({\left[\HL^i(\bmm)\right]}_j\right) \le \max_{1\le l \le n} \Big\{ \dim\left({\left[\HL^i\left(\left(\AAA/\pp_l\right)(a_l,b_l)\right)\right]}_j \right) \Big\}.		
		\end{equation*}
		If $\pp_l \supseteq \nnn\AAA$ then  $\AAA/\pp_l$ is a quotient of $\AAA/\nnn\AAA\cong R$ and this implies that 
		$$
		{\left[{\left[\HL^i\left(\AAA/\pp_l\right)\right]}_j\right]}_k={\left[\HL^i\left(\AAA/\pp_l\right)\right]}_{j,k}=0
		$$ 
		for $k\neq 0$.
		Thus, one can assume that $\pp_l \not\supseteq \nnn\AAA$.
		
	Alongside with the previous reductions, one can then assume that $\bmm=\AAA/\pp$ where $\pp$ is a bigraded prime ideal and $\pp \not\supseteq \nnn\AAA$.
		In this case there exists an homogeneous element $y \in {\left[S\right]}_1$ such that $y \not\in \pp$. 
		The short exact sequence 
		\begin{equation*}
			0 \rightarrow \left(\AAA/\pp\right)(0,-1) \xrightarrow{y}  \AAA/\pp \rightarrow  \AAA/(y,\pp) \rightarrow 0
		\end{equation*}		
		yields the long exact sequence in local cohomology 
		\begin{equation*}
			{\left[\HL^{i-1}\left(\AAA/(y,\pp)\right)\right]}_j \rightarrow \left({\left[\HL^{i}\left(\AAA/\pp\right)\right]}_j\right)(-1) \xrightarrow{y} {\left[\HL^{i}\left(\AAA/\pp\right)\right]}_j \rightarrow {\left[\HL^{i}\left(\AAA/(y,\pp)\right)\right]}_j. 
		\end{equation*}
		Therefore, it follows that 
		\begin{equation*}
			\dim\left({\left[\HL^i(\AAA/\pp)\right]}_j\right) \le \max\Big\{ \dim\left({\left[\HL^{i-1}\left(\AAA/(y,\pp)\right)\right]}_j\right), 1+ \dim\left({\left[\HL^{i}\left(\AAA/(y,\pp)\right)\right]}_j\right)\Big\}.
		\end{equation*}
		Since $\dim\left(\AAA/(y,\pp)\right)\le d-1$, the induction hypothesis gives
		$$
		\dim\left({\left[\HL^{i-1}\left(\AAA/(y,\pp)\right)\right]}_j\right) \le (d-1)-(i-1) = d-i
		$$
		and 
		$$
		1+\dim\left({\left[\HL^{i}\left(\AAA/(y,\pp)\right)\right]}_j\right) \le 1+(d-1)-i=d-i.
		$$
		Therefore, $\dim\left({\left[\HL^i(\AAA/\pp)\right]}_j\right)\le d-i$, as meant to be shown.
		
		(ii) 
		For $i\ge 1$, the isomorphism $\HH^i\left(\Proj_{R\text{-gr}}(\AAA), \widetilde{\bmm}(j)\right) \simeq {\left[\HL^{i+1}(\bmm)\right]}_j$ (see, e.g., \cite[Corollary 1.5]{HYRY_MULTIGRAD}, \cite[Theorem A4.1]{EISEN_COMM}) and part (i) imply the result.
		So, $i=0$ is the only remaining case.
		
		Let $\mathbb{M}^\prime=\mathbb{M}/\HH_\mm^0(\mathbb{M})$ and $\AAA^\prime = \AAA / \Ann_\AAA(\mathbb{M}^\prime)$.
		One has the following short exact sequence
		$$
		0 \rightarrow {\left[\mathbb{M}^\prime\right]}_j \rightarrow \HH^0\left(\Proj_{R\text{-gr}}(\AAA), \widetilde{\bmm}(j)\right) \rightarrow {\left[\HL^1(\bmm)\right]}_j \rightarrow 0.
		$$
		(see, e.g., \cite[Corollary 1.5]{HYRY_MULTIGRAD}, \cite[Theorem A4.1]{EISEN_COMM}).
		From part (i), one has the inequality $\dim\left({\left[\HL^1(\bmm)\right]}_j\right) \le \dim(\mathbb{M})-1$.
		
		So, it is enough to show that $\dim\left({\left[\mathbb{M}^\prime\right]}_j\right)\le \dim(\mathbb{M})-1$.
		If $\mathbb{M}^\prime=0$, the result is clear.
		Hence, assume that $\mathbb{M}^\prime\neq0$.
		Note that $\grade(\mm \AAA^\prime)>0$, and so  $\dim\left(\AAA^\prime/\mm \AAA^\prime\right)\le \dim(\AAA^\prime)-1 \le \dim\left(\mathbb{M}\right)-1$.
		Since $\AAA^\prime={\left[\AAA^\prime\right]}_0 \oplus \mm \AAA^\prime$, there is an isomorphism ${\left[\AAA^\prime\right]}_0 \simeq \AAA^\prime/\mm \AAA^\prime$ of graded $\kk$-algebras. 
		Therefore, the result follows because ${\left[\mathbb{M}^\prime\right]}_j$ is a finitely generated module over ${\left[\AAA^\prime\right]}_0$.
	\end{proof}
\end{theorem}

Of particular interest is the following corollary that generalizes \cite[Proposition 3.1]{MULTPROJ}.

\begin{corollary}
	Let $(R,\mm)$ be a standard graded algebra over a field with graded irrelevant ideal $\mm$.
	For any ideal $I\subset \mm$ one has 
	$$
	\dim\left({\left[\HL^i(\Rees_R(I))\right]}_j\right) \le \dim(R)+1-i
	$$ 
	and 
	$$
	\dim\left(\HH^i\left(\Proj_{R\text{-gr}}\left(\Rees_R(I)\right),\OO_{{\Proj_{R\text{-gr}}(\Rees_R(I))}}(j)\right)\right) \le \dim(R)-i
	$$
	for any $i\ge 0, j \in \ZZ$.
	\begin{proof}
		It follows from \autoref{thm:dim_cohom_bigrad} and the fact that $\dim\left(\Rees_R(I)\right)\le\dim(R)+1$.
	\end{proof}
\end{corollary}

\section{Specialization}
\label{sec:specialization}

In this section one studies how the process of specializing Rees algebras and saturated fiber cones affects the degree of specialized rational maps, where the latter is understood in terms of coefficient specialization.

The following notation will take over throughout this section.

\begin{setup}\label{NOTATION_GENERAL1}\rm
	Essentially keep the basic notation as in \autoref{sec:rat_maps}, but this time around take $A=\kk[z_1,\ldots,z_m]$ to be a polynomial ring over a field $\kk$ (for the present purpose forget any grading).
	Consider a rational map  $\GG :  \PP_A^r \dashrightarrow \PP_A^s$ given by a representative $\mathbf{g}=(g_0:\cdots:g_s)$, where $\PP_A^r={\rm Proj}(R)$ with $R= A[\xx]=A[x_0,\ldots,x_r]$.
	Fix a maximal ideal $\nnn=(z_1-\alpha_1, \ldots, z_m-\alpha_m)$ of $A$ where $\alpha_i \in \kk$.
	
	Clearly, one has $\kk\simeq A/\nnn$.
	One views the structure of $A$-module of $\kk$ as coming from the natural homomorphism $A \twoheadrightarrow  A/\nnn \simeq \kk$.
	Thus, one gets
	$$
	X \times_{A} \kk \simeq X \times_{A} (A/\nnn) \qquad\text{ and } \qquad M \otimes_{A} \kk \simeq M \otimes_{A} (A/\nnn),
	$$
	for any $A$-scheme $X$ and any $A$-module $M$.
	
	Then $R/\nnn R \simeq \kk[x_0,\ldots,x_r]$.
	Let $\bgg$ denote the rational map $\bgg: \PP_\kk^r \dashrightarrow \PP_\kk^s$ with representative
	$\overline{\mathbf{g}}=(\overline{g_0}:\cdots:\overline{g_s})$,
	where $\overline{g_i}$ is the image of $g_i$ under the natural map $R \twoheadrightarrow R/\nnn R$.
Further assume that $\overline{g_i} \neq 0$ for all $0\le i \le s$.

	Finally, denote $\I:=(g_0,\ldots,g_s) \subset R$ and $I:=(\I,\nnn)/\nnn
	=(\overline{g_0},\ldots,\overline{g_s}) \subset R/\nnn R$.
\end{setup}

	Recall that if $X$ stands for a topological space, a property of the points of $X$ is said to hold generically on $X$ provided it holds for every point in a dense open subset of $X$.

\begin{remark}\rm
	\label{rem:zariski_topology_affine}
	In the above setup, suppose in addition that $\kk$ is algebraically closed. 
	One can go back and forth between $\kk^m$ and a naturally corresponding subspace of $\Spec(A)$, with $A=\kk[z_1,\ldots,z_m]$.
	Namely, consider 
	$
	\MaxSpec(A):=\big\{ \pp \in \Spec(A) \mid \pp \text{ is maximal}  \big\}
	$
	with the induced Zariski topology of $\Spec(A)$.
	Then, by Hilbert's Nullstellensatz, the natural association 
	$(a_1,\ldots,a_m)\mapsto (z_1-a_1,\ldots,z_m-a_m)$
	yields a homeomorphism of  $\kk^m$ onto $\MaxSpec(A)$.
	In particular, it preserves dense open subsets and a dense open subset in $\Spec(A)$ restricts to a dense open subset of $\MaxSpec(A)$ naturally identified with a dense open subset  of $\kk^m$.
	This implies that if a property holds generically on points of $\Spec(A)$ then it holds generically on points of $\kk^m$.
\end{remark}

\subsection{Algebraic lemmata}

In this part one establishes the main algebraic relations between the objects introduced so far.
First is the picture of the main homomorphisms which is better envisaged through a commutative diagram.

\begin{lemma}\label{Rees-fiber_diagram}
	With the notation introduced above, one has a commutative diagram 
	\begin{center}
		\begin{tikzpicture}
		\matrix (m) [matrix of math nodes,row sep=1.8em,column sep=12.5em,minimum width=2em, text height=1.5ex, text depth=0.25ex]
		{
			A[\mathbf{g}] & \Rees_R(\I)\\
			A[\mathbf{g}]  \otimes_A \kk &  \Rees_R(\I) \otimes_A \kk\\
			\kk[\mathbf{\overline{g}}] &  \Rees_{ R/\nnn R}(I). \\
		};
		\path [draw,->>] (m-1-1) -- (m-2-1);
		\path [draw,->>] (m-2-1) -- (m-3-1);
		\path [draw,->>] (m-1-2) -- (m-2-2);
		\path [draw,->>] (m-2-2) -- (m-3-2);
		\draw[right hook->] (m-1-1)--(m-1-2);			
		\draw[right hook->] (m-2-1)--(m-2-2);			
		\draw[right hook->] (m-3-1)--(m-3-2);
		\end{tikzpicture}	
	\end{center}
	where $A[\mathbf{g}]$  {\rm (}respectively, $\kk[\overline{\mathbf{g}}]${\rm )} is identified with $\fib_R(\I)$ {\rm (}respectively, $\fib_{ R/\nnn R}(I)${\rm )}.
\end{lemma}
\begin{proof}
	The upper vertical maps are obvious surjections as $\kk=A/\nnn$, hence the upper square is commutative -- the lower horizontal map of this square is injective because in the upper horizontal map $A[\mathbf{g}]$ is injected as a direct summand.
	The right lower vertical map is naturally induced by the natural maps
	$$R[t] \surjects R[t]\otimes_A\kk=R[t]/\nnn R[t]=A[\mathbf{x}][t]/\nnn A[\mathbf{x}][t]\simeq A/\nnn[\mathbf{x}][t]=\kk[\mathbf{x}][t],
	$$
	where $t$ is a new indeterminate.
	The left lower vertical map is obtained by restriction thereof.
\end{proof}

\begin{notation}
	For a standard bigraded ring $T$, a bigraded $T$-module $M$ and a bihomogeneous prime $\mathfrak{P} \in \biProj(T)$, the bihomogeneous localization at $\mathfrak{P}$ is given and denoted by 
	$$
	M_{(\mathfrak{P})} := \Big\lbrace \frac{m}{g} \mid m \in M, \; g \in T \setminus \mathfrak{P}, \; \bideg(m) = \bideg(g) \Big\rbrace. 
	$$
\end{notation}

Further details about the above homomorphisms are established in the following proposition.

\begin{proposition}
	\label{prop:minimal_primes_tensor_Rees}
	Consider the  homomorphism of bigraded algebras $\mathfrak{s}:\Rees_R(\mathcal{I}) \otimes_A \kk\surjects \Rees_{ R/\nnn R}(I)$ as in {\rm \autoref{Rees-fiber_diagram}}, where $I\neq 0$. Then:
	\begin{enumerate}[\rm (i)]
		\item  $\ker(\mathfrak{s}) $ is a minimal prime ideal  of $\Rees_R(\mathcal{I}) \otimes_A \kk$ and, for any minimal prime $\mathfrak{Q}$ of $\Rees_R(\mathcal{I}) \otimes_A \kk$ other than $\ker(\mathfrak{s}) $, one has that $\mathfrak{Q}$ corresponds to a minimal prime of $\gr_{\mathcal{I}}(R) \otimes_A \kk \simeq \left(\Rees_ R(\mathcal{I}) \otimes_A \kk\right)/ \mathcal{I} \left(\Rees_ R(\mathcal{I}) \otimes_A \kk\right)$ and so 
		$$
		\dim\left((\Rees_ R(\mathcal{I}) \otimes_A \kk)/\mathfrak{Q}\right) \le \dim(\gr_{\mathcal{I}}( R) \otimes_A \kk).
		$$
		In particular,
		$
	\dim\left(\Rees_ R(\mathcal{I}) \otimes_A \kk \right)=\max\{r+2, \dim(\gr_{\mathcal{I}}( R) \otimes_A \kk)\}.
		$
	
		\item Let $k\ge 0$ be an integer such that $\ell(\mathcal{I}_{\mathfrak{P}}) \le \HT(\mathfrak{P}/\nnn R) + k$ for every prime ideal $\mathfrak{P} \in \Spec( R)$ containing $(\mathcal{I},\nnn)$.
		Then 
		$ \dim(\gr_{\mathcal{I}}( R) \otimes_A \kk)\le \dim (R/\nnn R)+k.
		$
		In particular,
		$
		\dim\left(\Rees_ R(\mathcal{I}) \otimes_A \kk \right)\le  \max\{r+2,r+k+1\}.
		$

		\item $\dim\left(\ker(\mathfrak{s})\right) \le \dim\left(\gr_{\I}(R)\otimes_A \kk\right)$.
		
		\item After localizing at $\ker(\mathfrak{s}) \in \biProj\left(\Rees_R(\mathcal{I}) \otimes_A \kk\right)$ one gets the natural isomorphism 
		$$
		{\left(\Rees_R(\mathcal{I}) \otimes_A \kk\right)}_{\large(\ker(\mathfrak{s})\large)} \,\xrightarrow{\simeq}\, {\left(\Rees_{ R/\nnn R}(I)\right)}_{\large(\ker(\mathfrak{s})\large)},
		$$
		where the right hand side is a localization as a module over $\Rees_R(\mathcal{I}) \otimes_A \kk$.
	\end{enumerate}
\end{proposition}
\begin{proof}
	(i)	Let $P\in \Spec (R)$ be a prime ideal not containing $\mathcal{I}$.
	Localizing the surjection $\mathfrak{s}:\Rees_R(\mathcal{I}) \otimes_A \kk\surjects \Rees_{ R/\nnn R}(I)$ at $R\setminus P$, one easily sees that it becomes an isomorphism.
	It follows that some power of $\mathcal{I}$ annihilates $\ker({\mathfrak{s}})$, that is 
	\begin{equation}
		\label{eq:vanishing_ker_specialization}
		\I ^l \cdot \ker(\mathfrak{s})=0
	\end{equation}	
	for some $l>0$.
	Since $I\neq 0$, one has $\mathcal{I} \not\subset \ker(\mathfrak{s})$.
	Therefore, any prime ideal of $\Rees_R(\mathcal{I}) \otimes_A \kk$ contains either the prime ideal $\ker{\mathfrak{(s)}}$ or the ideal $\mathcal{I}$.
	Thus, $\ker({\mathfrak{s}})$ is a minimal prime and any other minimal prime  $\mathfrak{Q}$ of $\Rees_R(\mathcal{I}) \otimes_A \kk$ contains  $\mathcal{I}$.
	Clearly, then  any such $\mathfrak{Q}$ is a minimal prime of $\left(\Rees_ R(\mathcal{I}) \otimes_A \kk\right)/ \mathcal{I} \left(\Rees_ R(\mathcal{I}) \otimes_A \kk\right) \simeq \gr_{\mathcal{I}}(R) \otimes_A \kk$.
	Since $\dim\left(\Rees_{ R/\nnn R}(I)\right)=\dim\left(R/\nnn R\right)+1$, the claim follows.
	
	(ii)
	For this, let $\mathfrak{M}$ be a minimal prime of $\gr_{\mathcal{I}}( R) \otimes_A \kk$ of maximal dimension, i.e.: 
	$$
	\dim(\gr_{\mathcal{I}}( R) \otimes_A \kk) = \dim((\gr_{\mathcal{I}}( R) \otimes_A \kk)/ \mathfrak{M}),
	$$
	and let $\mathfrak{P} = \mathfrak{M} \cap  R$ be its contraction to $ R$. 
	Clearly, $\mathfrak{P} \supseteq (\mathcal{I}, \nnn)$.
	By \cite[Lemma 1.1.2]{simis1988krull} and the hypothesis,
	\begin{align*}
	\dim\left(\gr_{\mathcal{I}}( R) \otimes_A \kk\right) &= \dim\big((\gr_{\mathcal{I}}( R) \otimes_A \kk)/ \mathfrak{M}\big) \\
	&= \dim( R/\mathfrak{P}) + \text{trdeg}_{R/\mathfrak{P}} \big((\gr_{\mathcal{I}}( R) \otimes_A \kk)/ \mathfrak{M}\big) \\
	&= \dim(R/\mathfrak{P}) + \dim \big(\left((\gr_{\mathcal{I}}( R) \otimes_A \kk)/ \mathfrak{M}\right)\otimes_{ R/\mathfrak{P}}   R_\mathfrak{P}/\mathfrak{P} R_\mathfrak{P} \big)\\
	&\le \dim( R/\mathfrak{P}) + \dim \big(\gr_{\mathcal{I}}( R) \otimes_{ R}   R_\mathfrak{P}/\mathfrak{P} R_\mathfrak{P} \big)\\
	&= \dim( R/\mathfrak{P}) + \ell(\mathcal{I}_{\mathfrak{P}})\\
	&\le \dim( R/\mathfrak{P})+\HT(\mathfrak{P}/\nnn)+k\\
	&\le \dim({ R/\nnn R}) + k,
	\end{align*}
	as required.
	
	\smallskip		
	
	(iii) This is clear since, by \autoref{eq:vanishing_ker_specialization}, one has $\Ann_{\Rees_R(\mathcal{I}) \otimes_A \kk}\left(\ker(\mathfrak{s})\right) \supset \mathcal{I}^l$.
	
	\smallskip 
	
	(iv) Let $K \subset \Rees_R(\mathcal{I}) \otimes_A \kk$ be the multiplicative set of bihomogeneous elements not in $\ker(\mathfrak{s})$.
	Since $\I^l \subset \Ann_{\Rees_R(\mathcal{I}) \otimes_A \kk}\left(\ker(\mathfrak{s})\right)$ and $\I^l \not\subset \ker(\mathfrak{s})$, it follows that $K^{-1}\ker(\mathfrak{s})=0$, and so the localized map $K^{-1}{\left(\Rees_R(\mathcal{I}) \otimes_A \kk\right)} \xrightarrow{\simeq} K^{-1}{\left(\Rees_{ R/\nnn R}(I)\right)}$ is an isomorphism.
	Finally, one has that ${\left(\Rees_R(\mathcal{I}) \otimes_A \kk\right)}_{\large(\ker(\mathfrak{s})\large)} = {\left[K^{-1}{\left(\Rees_R(\mathcal{I}) \otimes_A \kk\right)}\right]}_{(0,0)}$ and ${\left(\Rees_{ R/\nnn R}(I)\right)}_{\large(\ker(\mathfrak{s})\large)} = {\left[K^{-1}{\left(\Rees_{ R/\nnn R}(I)\right)}\right]}_{(0,0)}$.
\end{proof}

The next two lemmas are consequences of the primitive element theorem and the upper semi-continuity of the dimension of fibers, respectively.  Proofs are given for the reader's convenience.

\begin{lemma}
	\label{lem:upper_bound_deg}
	Let $\kk$ denote a field of characteristic zero and let ${C}\subset {B}$ stand for a finite extension of finitely generated $\kk$-domains.
	Let $\mathfrak{b}\subset {B}$ be a prime ideal and set $\mathfrak{c}:=\mathfrak{b}\cap {C}\subset {C}$ for its contraction. If ${C}$ is integrally closed then one has
	$$
	\left[\Quot(B/\mathfrak{b}):\Quot(C/\mathfrak{c})\right] \le \left[\Quot({B}):\Quot({C})\right].
	$$
\end{lemma}
\begin{proof}
	Let $\{b_1,\ldots, b_c\}$ generate $B$ as a $C$-module.
	Setting $\overline{C}:=C/\mathfrak{c}\subset \overline{B}=B/\mathfrak{b}$ then the images $\{\overline{b_1},\ldots, \overline{b_c}\}$ generate then $\overline{B}$ as a  $\overline{C}$-module. 
	Since the field extensions $\Quot(B)|\Quot(C)$ and $\Quot(\overline{B})|\Quot(\overline{C})$ are separable, and since $\kk$ is moreover infinite, one can find elements $\lambda_1,\ldots, \lambda_c\in \kk$ such that $L:=\sum_{i=1}^c \lambda_ib_i\in B$ and $\ell:=\sum_{i=1}^c \lambda_i\overline{b_i}\in \overline{B}$ are respective primitive elements of the above extensions.
	Let $X^u+a_1X^{u-1}+ \cdots+a_u=0$ denote the minimal polynomial of $L$ over $\Quot(C)$.
	Since $C$ is integrally closed, then $a_i \in C$ for all $1 \le i \le u$ (see, e.g., \cite[Theorem 9.2]{MATSUMURA}).
	Reducing modulo $\mathfrak{b}$, one gets  $\ell^u+\overline{a_1}\ell^{u-1}+ \cdots+\overline{a_u}=0.$
	Then the degree of the minimal polynomial of $\ell$ over $\Quot(\overline{C})$ is at most $u$, as was to be shown.
\end{proof}

\begin{lemma}
	\label{lem_dim_general_fiber}
	The set  
	$
	\big\lbrace \pp \in \Spec(A) \mid \dim\big(\gr_{\mathcal{I}}(R) \otimes_{A} k(\pp)\big) \le r+1 \big\rbrace
	$
	is open and dense in $\Spec(A)$.
	\begin{proof}		
		One can consider the associated graded ring $\gr_{\I}(R)=\bigoplus_{n\ge 0} \I^{n}/\I^{n+1}$ as a finitely generated single-graded $A$-algebra with zero graded part equal to $A$. 
		Therefore \cite[Theorem 14.8 b.]{EISEN_COMM} implies that for every $n$ the subset 
		$
		Z_n = \lbrace \pp \in \Spec(A) \mid \dim\big(\gr_{\mathcal{I}}(R) \otimes_{A} k(\pp)\big) \le n \rbrace
		$
		is open in $\Spec(A)$.
		
		Let $\KK=\Quot(A)$, $\mathbb{T}=R\otimes_{A}\KK=\KK[x_0,\ldots,x_r]$ and $\mathbb{I}=\mathcal{I} \otimes_{A} \KK$.
		The generic fiber of $A \hookrightarrow \gr_{\I}(R)$ is given by $\gr_{\I}(R) \otimes_A \KK = \gr_{\mathbb{I}}(\mathbb{T})$.
		Since $\dim\left(\gr_{\mathbb{I}}(\mathbb{T})\right)=\dim(\mathbb{T})=r+1$, then $(0) \in Z_{r+1}$ and so $Z_{r+1}\neq \emptyset$.
	\end{proof}
\end{lemma}

\subsection{Geometric picture}

In this part one stresses the corresponding geometric environment, along with some additional notation.

\begin{setup}\label{geometric_setup}\rm
First, recall from \autoref{NOTATION_GENERAL1} that $R=A[x_0,\ldots,x_r]$ is a standard polynomial ring over $A$.
One sets $\PP_A^r=\Proj(R)$ as before.
In addition, one had $A=\mathbb{F}[z_1,\ldots,z_m]$ and $\nnn\subset A$ a given (rational) maximal ideal.

Let  $\GG$ and $\bgg$ be as in \autoref{NOTATION_GENERAL1}.
Denote by $\Proj(A[\mathbf{g}])$ and 
$\Proj(\kk[\mathbf{\overline{g}}])$ the images of $\GG$ and $\bgg$, respectively (see \autoref{image}).
Let $\mathbb{B(\mathcal{I})}:=\biProj(\Rees_R(\I))$ and $\mathbb{B}(I):=\biProj(\Rees_{R / \nnn R}(I))$ be the graphs of $\GG$ and $\bgg$, respectively (see \autoref{graph}).

Let $\mathbb{E}(\I):=\biProj\left(\gr_{\I}(R)\right)$ be the exceptional divisor of $\mathbb{B(\mathcal{I})}$.
\end{setup}

Consider the commutative diagrams 
\begin{equation}
	\label{eq:graph_Pi}
	\begin{tikzpicture}[baseline=(current  bounding  box.center)]
	\matrix (m) [matrix of math nodes,row sep=1.8em,column sep=13.5em,minimum width=2em, text height=1.5ex, text depth=0.25ex]
	{
		\mathbb{B}(\mathcal{I}) &  \\
		{\PP_A^r} &  \Proj(A[\mathbf{g}]) \\
	};
	\path[-stealth]
	(m-1-1) edge node [above]	{$\Pi$} (m-2-2)
	(m-1-1) edge node [left]  {$\Pi^\prime$} (m-2-1)
	;		
	\draw[->,dashed] (m-2-1)--(m-2-2) node [midway,above] {$\GG$};	
	\end{tikzpicture}	
\end{equation}
and 
\begin{equation}
	\label{eq:graph_pi}
	\begin{tikzpicture}[baseline=(current  bounding  box.center)]
	\matrix (m) [matrix of math nodes,row sep=1.8em,column sep=13.5em,minimum width=2em, text height=1.5ex, text depth=0.25ex]
	{
		\mathbb{B}(I) &  \\
		{\PP_{\kk}^r} &  {\Proj(\kk[\mathbf{\overline{g}}])} \\
	};
	\path[-stealth]
	(m-1-1) edge node [above]	{$\pi$} (m-2-2)
	(m-1-1) edge node [left]  {$\pi^\prime$} (m-2-1)
	;		
	\draw[->,dashed] (m-2-1)--(m-2-2) node [midway,above] {$\bgg$};	
	\end{tikzpicture}	
\end{equation}
where $\Pi^\prime$ and $\pi^\prime$ are the blowing-up structural maps, which are well-known to be birational (see, e.g., \cite[Section II.7]{HARTSHORNE}). 

Note, for the sake of coherence, that had one taken care of a full development of rational maps in the biprojective environment over the ground ring $A$, via a judicious use of the preliminary material developed on \autoref{sec:rat_maps}, the first of the above diagrams would be routinely verified.
Most of the presently needed material in the biprojective situation is more or less a straightforward extension of the projective one.
Thus, for example, the field of rational functions of the biprojective scheme $\mathbb{B}(\mathcal{I})$ is given by the bihomogeneous localization of $\Rees_R(\I)$ at the null ideal, that is
$$
K(\mathbb{B}(\mathcal{I})):={\Rees_R(\I)}_{(0)} = \Big\{ \frac{f}{g} \mid  f,g \in \Rees_R(\I),  \bideg(f)=\bideg(g), g\neq 0 \Big\}.
$$
Then, the {\em degree} of the morphism $\Pi$ (respectively, $\Pi^\prime$) is given by
$$
\left[K(\mathbb{B}(\mathcal{I})):K(\Proj(A[\mathbf{g}]))\right] \qquad \text{(respectively,  $\left[K(\mathbb{B}(\mathcal{I})):K({\PP_A^r})\right]$).}
$$ 
Likewise, one has:

\begin{lemma}
	The following statements hold:
	\begin{enumerate}[\rm (i)]
		\item $K(\mathbb{B}(\mathcal{I}))=K({\PP_{\kk}^r})$.
		\item $\Pi^\prime$ is a birational morphism. 
		\item $\deg(\Pi)=\deg(\GG)$.
	\end{enumerate}
	\begin{proof}
		(i) It is clear that $K({\PP_{\kk}^r}) \subset K(\mathbb{B}(\mathcal{I}))$.
		Let $f/g \in K(\mathbb{B}(\mathcal{I}))$ with $f,g \in {\left[\Rees_R(\I)\right]}_{\alpha,\beta}$, 
		then it follows that $f=pt^{\beta}$ and $g=p^\prime t^\beta$ where $p,p^\prime \in {\left[R\right]}_{\alpha+\beta d}$ and $d=\deg(g_i)$ (see \autoref{eq_bigrad_Rees_alg}).
		Thus,  $f/g=p/p^\prime \in R_{(0)}$ and so $K(\mathbb{B}(\mathcal{I})) \subset  K({\PP_{\kk}^r})$.
		
		(ii) The argument is like  the implication (B1) $\Rightarrow$ (B2) in \autoref{def:birational_maps}.
		Let $\eta$ denote the generic point of $\mathbb{B}(\I)$ and $\xi$ that of $\PP_A^r$. 
		From part (i), ${\left(\Pi^\prime\right)}_\eta^\sharp \,:\, \OO_{{\PP_A^r},\xi} \rightarrow \OO_{\mathbb{B}(\mathcal{I}),\eta}$ is an isomorphism.  
		Therefore, \cite[Proposition 10.52]{GORTZ_WEDHORN}  yields the existence of dense open subsets $U \subset \mathbb{B}(\mathcal{I})$ and $V \subset {\PP_A^r}$ such that the restriction ${\Pi^\prime\mid}_U:U \xrightarrow{\simeq} V$ is an isomorphism.
		
		(iii) This follows from (i) and the definitions.
	\end{proof}
\end{lemma}

Thus, one has as expected:  $\Pi$ and $\pi$  are generically finite morphisms if and only if $\GG$ and $\bgg$ are so, in which case one has 
$$
\deg(\GG) = \deg(\Pi) \quad \text{ and } \quad \deg(\bgg)=\deg(\pi).
$$

Next is a reformulation of \autoref{Rees-fiber_diagram} by taking the respective associated $\Proj$ and $\biProj$ schemes.		

\begin{lemma}
	\label{lem:diagram_closed_immersions}
There is a commutative diagram 
	\begin{equation}
		\label{diag:closed_immersions}
		\begin{tikzpicture}[baseline=(current  bounding  box.center)]
		\matrix (m) [matrix of math nodes,row sep=1.8em,column sep=12.5em,minimum width=2em, text height=1.5ex, text depth=0.25ex]
		{
			\mathbb{B}(I) & {\Proj(\kk[\mathbf{\overline{g}}])}\\
			\mathbb{B}(\mathcal{I}) \times_{A} \kk &  \Proj(A[\mathbf{g}]) \times_{A} \kk\\
			\mathbb{B}(\mathcal{I}) &  \Proj(A[\mathbf{g}]) \\
		};
		\path[-stealth]
		(m-1-1) edge node [above]	{$\pi$} (m-1-2)
		(m-2-1) edge node [above]	{$\Pi\times_{A} \kk$} (m-2-2)
		(m-3-1) edge node [above]	{$\Pi$} (m-3-2)
		(m-1-1) edge node [left] {$p_1$} (m-2-1)
		(m-1-2) edge node [right] {$q_1$} (m-2-2)
		(m-2-1) edge node [left] {$p_2$} (m-3-1)
		(m-2-2) edge node [right] {$q_2$} (m-3-2)
		;		
		\end{tikzpicture}	
	\end{equation}	
	where the statements below are satisfied:
	\begin{enumerate}[\rm (i)]
		\item $p_1$ and $q_1$ are closed immersions.
		\item $p_2$ and $q_2$ are the natural projections from the fiber products. 
	\end{enumerate} 
\end{lemma}

\begin{corollary}
	\label{cor_dim_blow_ups}
	The following statements hold:
	\begin{enumerate}
		\item[\rm (i)] $\dim\left(\mathbb{B}(\I)\right)=\dim(A)+r$.
		\item[\rm (ii)] $\dim\left(\mathbb{B}(I)\right)=r$.
	\end{enumerate}
	\begin{proof}
	(i)	It follows from \autoref{lem:dim_biProj} and the well-known formula $\dim\left(\Rees_R(\I)\right)=\dim(R)+1$.
	
	(ii) The proof is analogous to the one of (i).
	\end{proof}
\end{corollary}

\begin{lemma}
	\label{lem:irreduc_comp_Rees}
	Assuming that $I \neq 0$, the following statements hold:
	\begin{enumerate}[\rm (i)]		
		\item   $\mathbb{B}(I)$ is an irreducible component of $\mathbb{B}(\mathcal{I}) \times_{A} \kk$ and, for any irreducible component $\mathcal{Z}$ of $\mathbb{B}(\mathcal{I}) \times_{A} \kk$ other than $\mathbb{B}(I)$, one has that $\mathcal{Z}$ corresponds to an irreducible component of ${\mathbb{E}}(\I) \times_A \kk$ and so
		$
		\dim\left(\mathcal{Z}\right) \le  \dim\left({\mathbb{E}}(\I) \times_A \kk\right).
		$
		
		\item Let $k\ge 0$ be an integer such that $\ell(\mathcal{I}_{\mathfrak{P}}) \le \HT(\mathfrak{P}/\nnn R) + k$ for every prime ideal $\mathfrak{P} \in \Spec( R)$ containing $(\mathcal{I},\nnn)$.
		Then 
		$ \dim\left({\mathbb{E}}(\I) \times_A \kk\right)\le r+k-1.
		$
		In particular,
		$
		\dim\left(\mathbb{B}(\mathcal{I}) \times_{A} \kk\right)\le  \max\{r,r+k-1\}.
		$		
		
		\item Let $\mathfrak{u}$ be the generic point of $\mathbb{B}(I)$ and $W=\Spec(D) \subset \mathbb{B}(\I) \times_{A} \kk$ be an affine open subscheme such that $p_1(\mathfrak{u}) \in W$.
		The closed immersion $p_1$ yields the natural short exact sequence $0 \rightarrow \qqq \rightarrow D \rightarrow B \rightarrow 0$, where $\Spec(B)=p_1^{-1}(W)$.
		Then, one has $\qqq \not\in \Supp_{D}(\qqq)$.
	\end{enumerate}
\end{lemma}
\begin{proof}
	(i) and (ii) are clear from \autoref{cor:dim_mult_proj_sub_scheme},
	 and \autoref{prop:minimal_primes_tensor_Rees}(i),(ii).
	 
	 (iii) Note that the localized map $D_\qqq \rightarrow B_\qqq$ coincides with the map on stalks 
	 $$
	 D_\qqq = \OO_{\mathbb{B(\mathcal{I})} \times_A \kk, \, p_1(\mathfrak{u})} \,=\, 	{\left(\Rees_R(\mathcal{I}) \otimes_A \kk\right)}_{\large(\ker(\mathfrak{s})\large)} \,\rightarrow\, {\left(\Rees_{ R/\nnn R}(I)\right)}_{\large(\ker(\mathfrak{s})\large)} \,=\, \OO_{\mathbb{B}(I),\, \mathfrak{u}} \,=\, B_\qqq.
	 $$
	 Then, \autoref{prop:minimal_primes_tensor_Rees}(iv) implies that $\qqq_{\qqq} = 0$. 
\end{proof}

\subsection{Main specialization result}

The following proposition reduces the main features of \autoref{diag:closed_immersions} to an affine setting over which $\Pi$ and $\pi$ are finite morphisms. In the following result one keeps the notation employed in \autoref{diag:closed_immersions}.

\begin{proposition}
	\label{prop:find_open_set}
	Under \autoref{geometric_setup}, assume that both $\GG$ and $\bgg$ are generically finite. 
	Referring to \autoref{diag:closed_immersions}, let $\xi$ be the generic point of $\Proj(\kk[\mathbf{\overline{g}}])$ and set $w = q_2(q_1(\xi)) \in \Proj(A[\mathbf{g}])$.
	Then, the following statements are satisfied:
	\begin{enumerate}[\rm (i)]
		\item 
		$V=\{y \in \Proj(A[\mathbf{g}]) \mid \Pi^{-1}(y) \text{ is a finite set} \}$ is a dense open subset in $\Proj(A[\mathbf{g}])$ and the restriction $\Pi^{-1}(V) \rightarrow V$ is a finite morphism.
		
		\item 
		If $\dim\left(\mathbb{E}(\I) \times_{A} \kk \right) \le r$,  then there exists an affine open subset $U = \Spec(\mathfrak{C}) \subset \Proj(A[\mathbf{g}])$ containing $w \in U$ and one obtains a commutative diagram of affine schemes 
		\begin{equation*}
		\label{diag:affine_restrictions}
		\begin{tikzpicture}[baseline=(current  bounding  box.center)]
		\matrix (m) [matrix of math nodes,row sep=1.8em,column sep=10em,minimum width=2em, text height=1.5ex, text depth=0.25ex]
		{
			\Spec(B) = p_1^{-1}(p_2^{-1}(\Pi^{-1}(U))) &  \Spec(C) = q_1^{-1}(q_2^{-1}(U)) \\
			\Spec(\mathfrak{B} \otimes_{A} \kk) = p_2^{-1}(\Pi^{-1}(U)) & \Spec(\mathfrak{C} \otimes_{A} \kk) = q_2^{-1}(U)\\
			\Spec(\mathfrak{B}) = \Pi^{-1}(U) &  U = \Spec(\mathfrak{C}) \\
		};
		\path[-stealth]
		(m-1-1) edge node [above]	{${\pi\mid}_{\Spec\left(B\right)}$} (m-1-2)
		(m-3-1) edge node [above]	{${\Pi\mid}_{\Spec\left(\mathfrak{B}\right)}$} (m-3-2)
		(m-2-1) edge node [above]	{${\Pi\mid}_{\Spec\left(\mathfrak{B}\right)} \times_{A} \kk$} (m-2-2)
		(m-1-1) edge (m-2-1)
		(m-1-2) edge (m-2-2)
		(m-2-1) edge (m-3-1)
		(m-2-2) edge (m-3-2)
		;		
		\end{tikzpicture}	
		\end{equation*}	
		where ${\pi\mid}_{\Spec\left(B\right)}$ and ${\Pi\mid}_{\Spec\left(\mathfrak{B}\right)}$ are finite morphisms.
		
		\item If $\dim\left(\mathbb{E}(\I) \times_{A} \kk \right) \le r-1$, one has the isomorphisms 
		$$
		\Pi^{-1}(w) \simeq \Spec\big(\mathfrak{B} \otimes_{\mathfrak{C}} \Quot(C)\big) \simeq \Spec\big(B \otimes_{C} \Quot(C)\big) \simeq \pi^{-1}(\xi).
		$$
	\end{enumerate}
\end{proposition}
\begin{proof}
	(i)
	Clearly, $\Pi$  is a projective morphism, hence is a proper morphism. 
	Thus, as a consequence of Zariski's Main Theorem (see \cite[Corollary 12.90]{GORTZ_WEDHORN}), the set 
	$$
	V=\{y \in \Proj(A[\mathbf{g}]) \mid \Pi^{-1}(y) \text{ is a finite set} \}
	$$ 
	is open in $\Proj(A[\mathbf{g}])$ and the restriction $\Pi^{-1}(V) \rightarrow V$ is a finite morphism.
	Since $\Pi$ is generically finite,   $V$ is nonempty (see, e.g.,  \cite[Exercise II.3.7]{HARTSHORNE}).
	
	\smallskip
	
	(ii) Since $q_1$ and $q_2$ are closed immersions,  one obtains $k(\xi) \simeq k(w)$ and a closed immersion
	$$
	\Pi^{-1}\left(\Proj\left(\kk[\mathbf{\overline{g}}]\right)\right) \hookrightarrow \Pi^{-1}\left(\Proj(A[\mathbf{g}]) \times_{A} \kk\right) = \mathbb{B(\mathcal{I})} \times_{\Proj(A[\mathbf{g}])} \left(\Proj(A[\mathbf{g}]) \times_{A} \kk\right) = \mathbb{B}(\mathcal{I}) \times_{A} \kk.
	$$
	Also, note that $\Pi^{-1}(w) = \mathbb{B}(\I) \times_{\Proj(A[\mathbf{g}])} k(w) \subset \mathbb{B}(\I) \times_{\Proj(A[\mathbf{g}])}  \Proj\left(\kk[\mathbf{\overline{g}}]\right) = \Pi^{-1}\left(\Proj\left(\kk[\mathbf{\overline{g}}]\right)\right)$.
	From \autoref{lem:irreduc_comp_Rees}(i), the assumption $\dim\left(\mathbb{E}(\I) \times_{A} \kk \right) \le r$, and the fact that $\bgg$ is generically finite, it follows that 
	$$
	\dim\left(\mathbb{B}(\mathcal{I}) \times_{A} \kk\right) =r=\dim\left({\Proj(\kk[\mathbf{\overline{g}}])}\right).
	$$
	So, the restriction map 
	$$\Psi
	: \Pi^{-1}\left(\Proj\left(\kk[\mathbf{\overline{g}}]\right)\right)  \rightarrow \Proj\left(\kk[\mathbf{\overline{g}}]\right)
	$$ is also generically finite, and the fiber $\Psi^{-1}(\xi) \simeq \Pi^{-1}(w) $ of $\Psi$ over $\xi$ is finite.
	Therefore, $w \in V$.
	
	Take the open set $V$ and shrink it down to an affine open subset  $U:=\Spec\left(\mathfrak{C}\right) \subset V$ such that $w \in U$.
	The scheme $q_1^{-1}(q_2^{-1}(U)) \neq \emptyset$ is also affine because $q_1$ and $q_2$ are affine morphisms (\autoref{lem:diagram_closed_immersions}). 
	Then set
	$
	q_1^{-1}(q_2^{-1}(U)) =: \Spec\left(C\right).
	$
	Since the restriction $\Pi^{-1}(U) \rightarrow U$ is a finite morphism,  $\Pi^{-1}(U)$ is also affine (see, e.g., \cite[Remark 12.10]{GORTZ_WEDHORN}, \cite[Exercise 5.17]{HARTSHORNE}). 
	Set $\Pi^{-1}(U)=:\Spec\left(\mathfrak{B}\right)$.
	Similarly,
	$
	p_1^{-1}(p_2^{-1}(\Pi^{-1}(U))) =: \Spec\left(B\right)
	$ 
	is also affine.
	It is clear that $\Spec(\mathfrak{B} \otimes_{A} \kk) = p_2^{-1}(\Pi^{-1}(U))$ and  $\Spec(\mathfrak{C} \otimes_{A} \kk) = q_2^{-1}(U)$.
	After the above considerations, one gets the claimed commutative diagram.
	
	\smallskip
	
	(iii)			
	Due to the diagram in part (ii), the isomorphisms 
	$$
	\Pi^{-1}(w) \simeq \Spec\big(\mathfrak{B} \otimes_{\mathfrak{C}} \Quot(C)\big) \quad \text{ and } \quad \pi^{-1}(\xi) \simeq \Spec\big(B \otimes_{C} \Quot(C)\big)
	$$ 
	follow.
	The surjections $\mathfrak{C} \twoheadrightarrow\mathfrak{C} \otimes_{A} \kk \twoheadrightarrow C$ yield that $\mathfrak{B} \otimes_{\mathfrak{C}} \Quot(C) \simeq \left(\mathfrak{B} \otimes_{A} \kk\right) \otimes_{\left(\mathfrak{C} \otimes_{A} \kk\right)} \Quot(C)$ and that $B \otimes_{\left(\mathfrak{C} \otimes_{A} \kk\right)} \Quot(C) \simeq B \otimes_{C} \Quot(C)$.
	One has the natural short exact sequence 
	\begin{equation}
	\label{eq:exact_seq_affine}
	0 \rightarrow \qqq \rightarrow \mathfrak{B} \otimes_{A} \kk \rightarrow B \rightarrow 0,
	\end{equation}
	and \autoref{lem:irreduc_comp_Rees}(iii) implies that
	\begin{equation}
	\label{eq:supp_qq}
	\qqq \not\in \Supp_{\mathfrak{B} \otimes_{A} \kk}(\qqq).					
	\end{equation}
	
	Note that $\dim(B) = \dim(\mathbb{B}(I))=r$.
	The assumption $\dim\left(\mathbb{E}(\I) \times_{A} \kk \right) \le r-1$ and \autoref{lem:irreduc_comp_Rees}(i) imply that $\dim\left((\mathfrak{B} \otimes_{A} \kk) / \pp\right) \le r-1$ for any $\pp \in \Spec(\mathfrak{B} \otimes_{A} \kk)$ different from $\qqq$.
	Since the map $\mathbb{g}$ is generically finite, one has $\dim(C)=\dim\left(\Proj(\kk[\mathbf{\overline{g}}])\right)=r$.
	It then follows that
	\begin{equation}
	\label{eq:supp_fiber_B_A_kk}
	\Supp_{\mathfrak{B} \otimes_{A} \kk}\left(\left(\mathfrak{B} \otimes_{A} \kk\right) \otimes_{\left(\mathfrak{C} \otimes_{A} \kk\right)} \Quot(C)\right) = \{\qqq\},
	\end{equation}
	because for any $\pp \in \Supp_{\mathfrak{B} \otimes_{A} \kk}\left(\left(\mathfrak{B} \otimes_{A} \kk\right) \otimes_{\left(\mathfrak{C} \otimes_{A} \kk\right)} \Quot(C)\right)$ one gets an inclusion $C \hookrightarrow (\mathfrak{B} \otimes_{A} \kk)/\pp$, thus implying $\dim\left((\mathfrak{B} \otimes_{A} \kk) / \pp\right) = r$.
	
	Applying the tensor product $- \otimes_{\left(\mathfrak{C} \otimes_{A} \kk\right)} \Quot(C)$ to \autoref{eq:exact_seq_affine} gives the exact sequence 
	$$
	\qqq \otimes_{\left(\mathfrak{C} \otimes_{A} \kk\right)} \Quot(C) \rightarrow \mathfrak{B} \otimes_{\mathfrak{C}} \Quot(C) \rightarrow B \otimes_C \Quot(C) \rightarrow 0.
	$$
	Then, \autoref{eq:supp_qq} and \autoref{eq:supp_fiber_B_A_kk} imply that $\qqq \otimes_{\left(\mathfrak{C} \otimes_{A} \kk\right)} \Quot(C) = 0$.
	Therefore, one obtains the required isomorphism $\Spec\big(\mathfrak{B} \otimes_{\mathfrak{C}} \Quot(C)\big) \simeq \Spec\big(B \otimes_{C} \Quot(C)\big)$.
\end{proof}

Next is the main result of this part.

\begin{theorem}
	\label{THM_REDUCTION_REES}
	Under \autoref{geometric_setup}, suppose that both $\GG$ and $\bgg$ are generically finite.
	\begin{enumerate}
		\item[{\rm (i)}] Assume that the following conditions hold:
		\begin{enumerate}[\rm (a)]
			\item \quad$\Proj(A[\mathbf{g}])$ is a normal scheme.
			\item \quad $\dim\left(\mathbb{E}(\I) \times_{A} \kk \right) \le r$.
			\item \quad $\kk$ is a field of characteristic zero.
		\end{enumerate}
		Then 
		$$
		\deg(\bgg) \le \deg(\GG).
		$$		
	\item[{\rm (ii)}] If $\dim\left(\mathbb{E}(\I) \times_{A} \kk \right) \le r-1$, then
		$$
		\deg(\bgg) \ge \deg(\GG).
		$$
	\item[\rm (iii)] Assuming that $\kk$ is algebraically closed, there is a dense open subset $\mathcal{W} \subset \kk^m$ such that, if $\nnn=(z_1-\alpha_1, \ldots, z_m-\alpha_m)$ with $(\alpha_1, \ldots, \alpha_m)\in \mathcal{W}$, one has
		$$
		\deg(\bgg) = \deg(\GG).
		$$ 
	\item[{\rm (iv)}] Consider the following condition:
	\newline $(I\kern-2.3pt\mathcal{K})$ $k\ge 0$ is a given integer such that $\ell(\mathcal{I}_{\mathfrak{P}}) \le \HT(\mathfrak{P}/{\nnn R}) + k$ for every prime ideal $\mathfrak{P} \in  \Spec(R)$ containing $(\mathcal{I},\nnn)$.
	
	Then:
	\begin{enumerate}
		\item[\rm  $(I\kern-2.3pt\mathcal{K}1)$] If $(I\kern-2.3pt\mathcal{K})$ holds with $k \le 1$, then condition {\rm (b)} of part {\rm (i)} is satisfied. 
		\item[\rm $(I\kern-2.3pt\mathcal{K}2)$] If $(I\kern-2.3pt\mathcal{K})$ holds with $k=0$, then the assumption of {\rm (ii)} is satisfied.
	\end{enumerate}	
	\end{enumerate}
\end{theorem}
	\begin{proof}
		(i)  Using condition (b), the diagram in \autoref{prop:find_open_set}(ii) corresponds to the following commutative diagram of ring homomorphisms 
		\begin{center}
			\begin{tikzpicture}
			\matrix (m) [matrix of math nodes,row sep=1.8em,column sep=12.5em,minimum width=2em, text height=1.5ex, text depth=0.25ex]
			{
				\mathfrak{C} &  \mathfrak{B} \\
				C & B\\
			};
			\path [draw,->>] (m-1-1) -- (m-2-1);
			\path [draw,->>] (m-1-2) -- (m-2-2);
			\draw[right hook->] (m-1-1)--(m-1-2);			
			\draw[right hook->] (m-2-1)--(m-2-2);			
			;		
			\end{tikzpicture}			
		\end{center}
		with finite horizontal maps, which 	are injective because ${\pi\mid}_{\Spec\left(B\right)}$ and ${\Pi\mid}_{\Spec\left(\mathfrak{B}\right)}$ are dominant morphisms and $B$ and $\mathfrak{B}$ are domains (see \cite[Corollary 2.11]{GORTZ_WEDHORN}).
		Since $\Proj(A[\mathbf{g}])$ is given to be a normal scheme then $\mathfrak{C}$ is integrally closed. 
		By \autoref{lem:upper_bound_deg},
		$$
		\deg(\bgg)=\deg(\pi)=\deg\left({\pi\mid}_{\Spec\left(B\right)}\right) \le \deg\left({\Pi\mid}_{\Spec(\mathcal{B})}\right)=\deg(\Pi)=\deg(\GG).
		$$
	
		(ii) 		
		Let $\eta$ be the generic point of $\Proj(A[\mathbf{g}])$.
		By \autoref{prop:find_open_set}(iii) one has that 
		$$
		\deg(\mathbb{g}) = \dim_{k(\xi)}\left(\OO(\pi^{-1}(\xi))\right) = \dim_{\Quot(C)}\left(\mathfrak{B} \otimes_\mathfrak{C} \Quot(C)\right)
		$$
		where $\mathfrak{B}$ is a finitely generated $\mathfrak{C}$-module and $C \simeq \mathfrak{C} / Q$ for some $Q \in \Spec(\mathfrak{C})$. 
		On the other hand, 
		$$
		\deg(\mathcal{G}) = \dim_{k(\eta)}\left(\OO(\Pi^{-1}(\eta))\right) = \dim_{\Quot(\mathfrak{C})}\left(\mathfrak{B} \otimes_\mathfrak{C} \Quot(\mathfrak{C})\right).
		$$
		Applying Nakayama's lemma one obtains
		$$
		\dim_{\Quot(C)}\left(\mathfrak{B} \otimes_\mathfrak{C} \Quot(C)\right) \ge \dim_{\Quot(\mathfrak{C})}\left(\mathfrak{B} \otimes_\mathfrak{C} \Quot(\mathfrak{C})\right)
		$$
		(see, e.g., \cite[Example III.12.7.2]{HARTSHORNE}, \cite[Corollary 7.30]{GORTZ_WEDHORN}).
		So, the result follows.

	(iii) It follows from \autoref{specializing_mul_fiber} and \autoref{cor:specialization_image}.
		
	(iv) Both  $(I\kern-2.3pt\mathcal{K}1)$ and $(I\kern-2.3pt\mathcal{K}2)$  follow from \autoref{lem:irreduc_comp_Rees}.
		\end{proof}

\subsection{Specialization of the saturated fiber cone and the fiber cone}

This part deals with the problem of specializing saturated fiber cones and fiber cones.
Under certain general conditions it will turn out that the multiplicities of the saturated fiber cone and of the fiber cone remain constant under specialization. 
An important consequence is that under a suitable general specialization of the coefficients, one shows that the degrees of the rational map and the corresponding image remain constant.

The reader is referred to the notation of \autoref{sectio_on_saturation}. 

\begin{setup}\rm \label{NOTA_SAT_SPECIAL_FIB}
	Keep the notation introduced in \autoref{NOTATION_GENERAL1} and \autoref{geometric_setup}.
	Suppose that $\kk$ is an algebraically closed field.
	Let $\KK:=\Quot(A)$ denote the field of fractions of $A$ and
	let $\mathbb{T} $ denote the standard graded polynomial ring $\mathbb{T}:=\KK[x_0,\ldots,x_r]$ over  $\KK$.
	Let  $\GG$ and $\bgg$ be as in \autoref{NOTATION_GENERAL1}.
	
	In addition, let $\bGG$ denote the rational map $\bGG: \PP_\KK^r \dashrightarrow \PP_\KK^s$ with representative $\mathbf{G}=(G_0:\cdots:G_s)$,
	where $G_i$ is the image of $g_i$ along the natural inclusion  $R \hookrightarrow\mathbb{T}$.
	Set $\mathbb{I}:=(G_0,\ldots,G_s) \subset \mathbb{T}$.		
	Finally, let $Y:=\Proj(\kk[\overline{\mathbf{g}}])$ and $\mathbb{Y}:=\Proj(\KK[\mathbf{G}])$ be the images of $\bgg$ and $\bGG$,  respectively (see \autoref{image}). 
\end{setup}

As in \autoref{rem:degree_rat_map_over_KK},
the rational map $\GG$ is generically finite if and only if the rational map $\bGG$ is so, and one has the equality
$
\deg(\GG)=\deg\left(\bGG\right).
$

 Consider the projective $R$-scheme $\Proj_{R\text{-gr}}(\Rees_R(\I))$, where $\Rees_R(\I)$ is viewed as a ``one-sided'' $R$-graded algebra. 

For any $\pp \in \Spec(A)$, let $k(\pp):=A_\pp/\pp A_\pp$.
The fiber $\Rees_R(\I) \otimes_A k(\pp)$ inherits a one-sided structure of a graded $R(\pp)$-algebra, where $R(\pp):=R_\pp/\pp R_\pp=k(\pp)[x_0,\ldots,x_r]$.
Moreover, it has a natural structure as a bigraded algebra over  $R(\pp)[y_0,\ldots,y_s]=R(\pp) \otimes_{A} A[y_0,\ldots,y_s]$.

Therefore, for $0 \le i \le r$ the sheaf cohomology 
\begin{equation}
	\label{eq:cohom_groups_tensored}
	{\mathcal{M}(\pp)}^i := \HH^i\left(\Proj_{R(\pp)\text{-gr}}\left(\Rees_R(\I) \otimes_A k(\pp) \right), \OO_{\Proj_{R(\pp)\text{-gr}}\left(\Rees_R(\I) \otimes_A k(\pp) \right)}\right)
\end{equation}
has a natural structure as a finitely generated graded $k(\pp)[y_0,\ldots,y_s]$-module (see, e.g., \cite[Proposition 2.7]{MULTPROJ}).
In particular, one can consider its Hilbert function 
$$
\mathcal{H}\left({\mathcal{M}(\pp)}^i,k\right):=\dim_{k(\pp)}\left({\left[{\mathcal{M}(\pp)}^i\right]}_k\right).
$$

\begin{lemma}
	\label{lem:open_dense_subset_sat_spec}
	For any given $\pp\in\Spec(A)$, consider the function $\chi_{\pp}:\ZZ \rightarrow \ZZ$ defined by
	$$
	\chi_\pp(k) := \sum_{i=0}^{r} {(-1)}^i \mathcal{H}\left({\mathcal{M}(\pp)}^i,k\right).
	$$
	Then, there exists a dense open subset $\mathcal{U} \subset \Spec(A)$, such that $\chi_{\pp}$	is the same for all $\pp \in \mathcal{U}$.
	\begin{proof}
		Consider the affine open covering  $$\mathcal{W}:={\left(\Spec\left({\Rees_R(\I)}_{(x_i)}\right)\right)}_{0\le i \le r}$$ 
		of $\Proj_{R\text{-gr}}(\Rees_R(\I))$, with corresponding \v{C}ech complex 
		$$
		C^\bullet(\mathcal{W}) : \quad 0\rightarrow \bigoplus_{i}  \Rees_R(\I)_{(x_i)} \rightarrow \bigoplus_{i<j} \Rees_R(\I)_{(x_ix_j)} \rightarrow \cdots \rightarrow \Rees_R(\I)_{(x_0\cdots x_r)} \rightarrow 0.
		$$
		Note that each $C^i(\mathcal{W})$ is a direct sum of homogeneous localizations $\Rees_R(\I)_{\left(x_{l_1} \cdots\; x_{l_i}\right)}$ with $0 \le l_1 < \cdots < l_i \le r$, and that each
		$$
		\Rees_R(\I)_{\left(x_{l_1} \cdots\; x_{l_i}\right)} = A\left[\Big\lbrace\frac{f}{x_{l_1} \cdots\; x_{l_i}} \mid f \in R \text{ and } \deg(f) = i \Big\rbrace,  g_0t, \ldots, g_st\right]
		$$  
		has a natural structure of finitely generated graded algebra over $A$, and the grading comes from the graded structure of $A[\yy] \rightarrow A[\I t], \, y_i \mapsto g_it$ (also, see \autoref{eq_one_side_grading}).
		By using the Generic Freeness Lemma (see, e.g., \cite[Theorem 14.4]{EISEN_COMM}), there exist elements $a_i \in A$ such that each graded component of the localization ${C^i(\mathcal{W})}_{a_i}$ is a free module over $A_{a_i}$.	
		
		Let $D^\bullet$ be the complex given by $D^i = C^i(\mathcal{W})_a$, where $a=a_0a_1\cdots a_r$.
		Hence, now $D^\bullet$ is a complex of graded $A_a[\yy]$-modules and each graded strand ${\left[D^\bullet\right]}_k$ is a complex of free $A_a$-modules. 
		Note that each of the free $A_a$-modules ${\left[D^i\right]}_k$ is almost never finitely generated.

		The $i$-th cohomology of a (co-)complex $F^\bullet$ is denoted by $\HH^i(F^\bullet)$.
		Since each ${\left[D^\bullet\right]}_k$ is a complex of free $A_a$-modules (in particular, flat), \cite[Lemma 1 in p.~47 and Lemma 2 in p.~49]{MUMFORD} yield the existence of complexes $L_k^\bullet$ of finitely generated $A_a$-modules such that 
		\begin{eqnarray}
			\label{eq:Grothendieck_complex}
			\begin{split}
			\HH^i\Big({\left[D^\bullet\right]}_k \otimes_{A_a} k(\pp)\Big) &\simeq \HH^i\Big(L_k^\bullet \otimes_{A_a} k(\pp)\Big)	,\\
			L_k^i \neq 0 \,\text{ only if }\, 0 \le i \le r, \quad	L_k^i \, &\text{ is $A_a$-free for } \, 1 \le i \le r \quad \text{ and } \quad L_k^0 \, \text{ is $A_a$-flat }
			\end{split}			
		\end{eqnarray}
		for all $\pp \in \Spec(A_a) \subset  \Spec(A)$.
		Let $\mathcal{U}:=\Spec(A_a) \subset \Spec(A)$. 
		
		{\sc Claim.}  $\chi_\pp$ is independent of $\pp$ on $ \mathcal{U}$; in other words, for any $\pp\in \mathcal{U}$ and any $\mathfrak{q}\in \mathcal{U}$, one has $\chi_{\pp}(k)=\chi_{\mathfrak{q}}(k)$ for every $k\in\ZZ$.
		
		Consider an arbitrary $\pp \in \mathcal{U}$.
		Since $\Rees_R(\I) \otimes_A k(\pp) \simeq {\Rees_R(\I)}_a \otimes_{A_a} k(\pp)$, then $D^\bullet \otimes_{A_a} k(\pp)$ coincides with the \v{C}ech complex corresponding with the affine open covering 
		$$
		{\left(\Spec\left({\left(\Rees_R(\I) \otimes_A k(\pp)\right)}_{(x_i)}\right)\right)}_{0\le i \le r}
		$$
		of $\Proj_{R(\pp)\text{-gr}}\left(\Rees_R(\I) \otimes_A k(\pp) \right)$.
		Hence, from \autoref{eq:cohom_groups_tensored} and \autoref{eq:Grothendieck_complex}, for any $k\in\ZZ$ there is an isomorphism
		$$
		{\left[{\mathcal{M}(\pp)}^i\right]}_k \simeq \HH^i\Big(L_k^\bullet \otimes_{A_a} k(\pp)\Big).
		$$ 
		But since each $L_k^i$ is a finitely generated flat $A_a$-module, hence $A_a$-projective, it follows that $\dim_{k(\pp)}\left(L_k^i \otimes_{A_a} k(\pp)\right) = \rank_{A_a}\left(L_k^i\right)$ (see, e.g., \cite[Exercise 20.13]{EISEN_COMM}) and that 
		$$
		\sum_{i=0}^{r}{(-1)}^i\dim_{k(\pp)}\left({\left[{\mathcal{M}(\pp)}^i\right]}_k\right) = \sum_{i=0}^{r}{(-1)}^i \rank_{A_a}\left(L_k^i\right).
		$$
		Therefore,  for every $k\in\ZZ$, $\chi_\pp(k) = \sum_{i=0}^{r} {(-1)}^i \mathcal{H}\left({\mathcal{M}(\pp)}^i,k\right)$ does not depend on $\pp$.  
	\end{proof}
\end{lemma}

The following theorem contains the main result of this part. By considering saturated fiber cones, one asks how the product of the degrees of the map and of its image behave under specialization.

\begin{theorem}\label{specializing_mul_fiber}
	Under {\rm \autoref{NOTA_SAT_SPECIAL_FIB}}, suppose that $\GG$ is generically finite.
	Then, there is a dense open subset $\mathcal{V} \subset \kk^m$ such that, if $\nnn=(z_1-\alpha_1, \ldots, z_m-\alpha_m)$ with $(\alpha_1, \ldots, \alpha_m)\in \mathcal{V}$, one has
		$$
		\deg(\bgg)\cdot \deg_{\PP_\kk^s}(Y)  = e\left(\widetilde{\mathfrak{F}_{R/\nnn R}(I)}\right) = e\left(\widetilde{\mathfrak{F}_{\mathbb{T}}(\mathbb{I})}\right) = \deg(\bGG)\cdot \deg_{\PP_\KK^s}(\mathbb{Y}).
		$$ 
	\begin{proof}
		Take a dense open subset $\mathcal{U}$ like the one of \autoref{lem:open_dense_subset_sat_spec}, then from \autoref{lem_dim_general_fiber} one has 
		$$
		\mathcal{V} = \mathcal{U} \,\cap\, 	\big\lbrace \pp \in \Spec(A) \mid \dim\big(\gr_{\mathcal{I}}(R) \otimes_{A} k(\pp)\big) \le r+1 \big\rbrace \cap \MaxSpec(A)
		$$
		is a dense open subset of $\MaxSpec(A)$.
		From \autoref{rem:zariski_topology_affine}, $\mathcal{V}$ induces a dense open subset in $\kk^m$.
		
		From now on, suppose that $\nnn \in \mathcal{V}$.
		
		Let $W:=\Proj_{R(\nnn)\text{-gr}}\left(\Rees_R(\I) \otimes_A \kk \right)$, as in \autoref{eq:cohom_groups_tensored},  $\HH^i(W,\OO_W)={\mathcal{M}(\nnn)}^i$.
		By a similar token, $\HH^i(\mathbb{W},\OO_\mathbb{W})={\mathcal{M}((0))}^i$, where $\mathbb{W}:=\Proj_{R(0)\text{-gr}}\left(\Rees_R(\I) \otimes_A \KK \right)$, with $R(0):=R\otimes_ A\KK=\KK[x_0,\ldots,x_r]$ and $(0)$ denotes the null ideal of $A$.
		
		Now, clearly $(0) \in \mathcal{V}$ and $\nnn \in \mathcal{V}$. Therefore, \autoref{lem:open_dense_subset_sat_spec} yields the equalities 
		\begin{equation}
		\label{eq:Hilbert_polys_W}
		\sum_{i=0}^{r} {(-1)}^i \mathcal{H}\Big(\HH^i(W,\OO_W), k\Big) = \sum_{i=0}^{r} {(-1)}^i \mathcal{H}\Big(\HH^i(\mathbb{W},\OO_\mathbb{W}), k\Big)
		\end{equation}
		for all $k\in \NN$.
		
		By the definition of $\mathcal{V}$ and \autoref{prop:minimal_primes_tensor_Rees} one gets that $\dim\left(\Rees_R(\I) \otimes_{A} \kk\right)=\dim(R/\nnn R)+1$.
		From \autoref{thm:dim_cohom_bigrad}(ii) (by taking $\Rees_R(\I) \otimes_A \kk$ or $\Rees_R(\I) \otimes_A \KK$ as a module over itself), for any $i\ge 1$, one has the inequalities
		$$
		\dim\left(\HH^i(W,\OO_W)\right)\le \dim\left(R/\nnn R\right)-1 \;\text{ and }\; \dim\left(\HH^i(\mathbb{W},\OO_\mathbb{W})\right)\le \dim(\mathbb{T})-1.
		$$
		Therefore, \autoref{eq:Hilbert_polys_W} gives that
		$
		\dim\left(\HH^0(W,\OO_W)\right) = \dim\left(\HH^0(\mathbb{W},\OO_\mathbb{W})\right) = \dim(\mathbb{T}) = \dim\left(R/\nnn R\right),
		$
		and that the leading coefficients of the Hilbert polynomials of $\HH^0(W,\OO_W)$ and  $\HH^0(\mathbb{W},\OO_\mathbb{W})$ coincide, and so $e\left(\HH^0\left(W, \OO_W\right)\right)=e\left(\HH^0\left(\mathbb{W}, \OO_\mathbb{W}\right)\right)$.

		Consider the exact sequence of finitely generated graded $\left(\Rees_R(\I) \otimes_A \kk\right)$-modules
		$$
		0 \rightarrow \ker(\mathfrak{s}) \rightarrow \Rees_R(\I) \otimes_A \kk \xrightarrow{\mathfrak{s}} \Rees_{R/\nnn R}(I) \rightarrow 0
		$$
		where $\mathfrak{s}:\Rees_R(\I) \otimes_A \kk \twoheadrightarrow \Rees_{R/\nnn R}(I)$ denotes the same natural map of \autoref{prop:minimal_primes_tensor_Rees}.
		
		Sheafifying and taking the long exact sequence in cohomology yield an exact sequence of finitely generated graded $\kk[\yy]$-modules
		$$
		0 \rightarrow \HH^0\left(W, \ker(\mathfrak{s})^\sim\right) \rightarrow \HH^0\left(W, \OO_W\right) \rightarrow \HH^0\left(W, {\Rees_{R / \nnn R}(I)}^\sim\right) \rightarrow \HH^1\left(W, \ker(\mathfrak{s})^\sim\right).
		$$
		Note that 
		\begin{align*}
			\widetilde{\fib_{R/\nnn R}(I)} &\simeq \HH^0\left(\Proj_{{R(\nnn)}\text{-gr}}\left(\Rees_{R/\nnn R}(I)\right), \OO_{\Proj_{{R(\nnn)}\text{-gr}}\left(\Rees_{R/\nnn R}(I)\right)}\right) \\
			&\simeq \HH^0\left(W, {\Rees_{R/\nnn R}(I)}^\sim\right)
		\end{align*}
		 (see, e.g., \cite[Lemma III.2.10]{HARTSHORNE}).

		From the definition of $\mathcal{V}$ and \autoref{prop:minimal_primes_tensor_Rees}(iii), it follows that $\dim(\ker(\mathfrak{s})) \le \dim(R/\nnn R)$.		
		Hence \autoref{thm:dim_cohom_bigrad}(ii) gives  
		$$
		\dim\left(\HH^0\left(W, \ker(\mathfrak{s})^\sim\right)\right) \le \dim(R/\nnn R)-1 \, \text{ and } \, \dim\left(\HH^1\left(W, \ker(\mathfrak{s})^\sim\right)\right) \le \dim(R/\nnn R)-2.
		$$
		Therefore, one gets the equality 
		$e\left(\widetilde{\fib_{R/\nnn R}(I)}\right) = e\left(\HH^0\left(W, \OO_W\right)\right)$.
		Since $\widetilde{\mathfrak{F}_{\mathbb{T}}(\mathbb{I})} \simeq \HH^0(\mathbb{W},\OO_\mathbb{W})$,
		summing up yields  
		$$
		e\left(\widetilde{\mathfrak{F}_{R/\nnn R}(I)}\right) = e\left(\HH^0\left(W, \OO_W\right)\right) = e\left(\HH^0\left(\mathbb{W}, \OO_\mathbb{W}\right)\right) = e\left(\widetilde{\mathfrak{F}_{\mathbb{T}}(\mathbb{I})}\right).
		$$
		Finally, by \cite[Theorem 2.4]{MULTPROJ} it follows that 
		$$
		e\left(\widetilde{\mathfrak{F}_{R/\nnn R}(I)}\right) = \deg(\bgg)\cdot \deg_{\PP_\kk^s}(Y)   \quad \text{ and } \quad e\left(\widetilde{\mathfrak{F}_{\mathbb{T}}(\mathbb{I})}\right) = \deg(\bGG)\cdot \deg_{\PP_\KK^s}(\mathbb{Y}),
		$$	
		and so the result is obtained.
	\end{proof}
\end{theorem}

The following lemma can be considered as a convenient adaptation of \cite[Theorem 1.5]{TRUNG_SPECIALIZATION} in the current bigraded setting.

As before, let $\AAA$ be the bigraded $A$-algebra $\AAA= R \otimes_A A[y_0,\ldots,y_s]$.

\begin{lemma}
	\label{lem_general_specialization_complexes}
	Let 
	$
	F_\bullet:  0 \rightarrow F_k \xrightarrow{\phi_k} \cdots \xrightarrow{\phi_2} F_1 \xrightarrow{\phi_1} F_0
	$
	be an acyclic complex of finitely generated free bigraded $\AAA$-modules.
	Then, there is a dense open subset $V \subset \kk^m$ such that, if $\nnn=(z_1-\alpha_1,\ldots,z_m-\alpha_m)$ with $(\alpha_1,\ldots,\alpha_m) \in V$, the complex 
	$$
	F_\bullet \otimes_A \kk: \qquad 0 \rightarrow F_k\otimes_A \kk \xrightarrow{\phi_k \otimes_A \kk} \cdots \xrightarrow{\phi_2\otimes_A \kk} F_1\otimes_A \kk \xrightarrow{\phi_1\otimes_A\kk} F_0\otimes_A \kk
	$$
	is acyclic. 
	\begin{proof}
		Clearly, the complex $F_\bullet \otimes_A \KK$, where $\KK=\Quot(A)$, is also acyclic.
		Denote by $\Phi_i$ the map 
		$$
		\Phi_i  = \phi_i \otimes_A \KK : F_i \otimes_A \KK \rightarrow F_{i-1} \otimes_A \KK
		$$
		for $1 \le i \le k$.
				
		From the Buchsbaum-Eisenbud exactness criterion (see, e.g., \cite[Theorem 20.9]{EISEN_COMM}, \cite{BUCHSBAUM_EISENBUD_EXACT}), it follows that $F_\bullet \otimes_A \kk$ is acyclic if $\rank(\phi_i \otimes_A \kk) = r_i$ and $\grade\left(I_{r_i}(\phi_i \otimes_A \kk)\right) = \grade\left(I_{r_i}(\Phi_i)\right)$,
		with $r_i=\rank(\Phi_i) = \rank(\phi_i)$ for $1 \le i \le k$.
		
		It is clear that $\rank(\phi_i\otimes_A \kk) \le r_i$, and the condition $\rank(\phi_i \otimes_A \kk) = r_i$ is satisfied when $(\alpha_1,\ldots,\alpha_m) \in \kk^m$ belongs to the dense open subset $U \subset \kk^m$ which is the complement of the closed subset of $\kk^m$ given as the vanishing of all the $A$-coefficients of the $r_i\times r_i$ minors of $\phi_i$.
		
		Since $\AAA \otimes \kk$ and $\AAA \otimes_A \KK$ are Cohen-Macaulay rings, one has that
		$\grade\left(I_{r_i}(\phi_i \otimes_A \kk)\right) = \HT\left(I_{r_i}(\phi_i \otimes_A \kk)\right)$ and $\grade\left(I_{r_i}(\Phi_i)\right) = \HT\left(I_{r_i}(\Phi_i)\right)$, and that it is enough to check 
		\begin{equation}
			\label{eq_dim_general_fiber_det_ideals}
			\dim\left(\frac{\AAA}{I_{r_i}(\phi_i)} \otimes_A \KK\right) = \dim\left(\frac{\AAA}{I_{r_i}(\phi_i)} \otimes_A \kk\right).
		\end{equation}
		From \cite[Theorem 14.8]{EISEN_COMM} and \autoref{rem:zariski_topology_affine}, there exists a dense open subset $U^\prime \subset \kk^m$ such that, if $(\alpha_1,\ldots,\alpha_m) \in U^\prime$, then the equality \autoref{eq_dim_general_fiber_det_ideals} is satisfied.
		
		Therefore, the result follows by setting $V = U \cap U^\prime$. 
	\end{proof}
\end{lemma}

\begin{lemma}
	\label{lem_general_specialization_Rees_algebra}
	There is a dense open subset $V \subset \kk^m$ such that, if $\nnn=(z_1-\alpha_1,\ldots,z_m-\alpha_m)$ with $(\alpha_1,\ldots,\alpha_m) \in V$, then the natural map
	$
	\mathfrak{s}:\Rees_R(\mathcal{I}) \otimes_A \kk\surjects \Rees_{ R/\nnn R}(I)
	$	
	is an isomorphism.
	\begin{proof}
		To prove that $\mathfrak{s}$ is an isomorphism, it is enough to show that 
		$
		\Tor_1^R\left(\gr_{\I}(R), R/\nnn R\right)=0;
		$
		see \cite[\S 2]{SIMIS_ULRICH_SPECIALIZATION}.
		
		Let 
		$
		F_\bullet : 0 \rightarrow F_k \rightarrow \cdots  \rightarrow F_0
		$
		be the minimal bigraded free resolution of $\gr_{\I}(R)$ as a finitely generated bigraded $\AAA$-module. 
		Let $V \subset \kk^m$ be a dense open subset as in \autoref{lem_general_specialization_complexes} and suppose that $(\alpha_1,\ldots,\alpha_m) \in \kk^m$.
		Thus, since $F_\bullet \otimes_\AAA \AAA/\nnn \AAA$ is an acyclic complex,  one has 
		$
		\Tor_1^\AAA\left(\gr_{\I}(R), \AAA/\nnn \AAA\right)=0.
		$
		 As $\Tor_1^R\left(\gr_{\I}(R), R/\nnn R\right)=\Tor_1^\AAA\left(\gr_{\I}(R), \AAA/\nnn \AAA\right)$, the result is clear.
	\end{proof}
\end{lemma}

Now one treats the problem of specializing the fiber cone. 
As a consequence, one shows that the degree of the image of a rational map is stable under specialization. 

\begin{theorem}
	\label{cor:specialization_image}
	Under {\rm \autoref{NOTA_SAT_SPECIAL_FIB}}, suppose that $\GG$ is generically finite.
	Then, there is a dense open subset $\mathcal{Q} \subset \kk^m$ such that, if $\nnn=(z_1-\alpha_1, \ldots, z_m-\alpha_m)$ with $(\alpha_1, \ldots, \alpha_m)\in \mathcal{Q}$, one has
	$$
		\deg_{\PP_\kk^s}(Y)  = e(\kk[\mathbf{\overline{g}}]) = e(\KK[\mathbf{G}]) = \deg_{\PP_\KK^s}(\mathbb{Y}).
	$$		
	\begin{proof}
		Let $V \subset \kk^m$ be the dense open subset of \autoref{lem_general_specialization_Rees_algebra} and assume that $\nnn \in V$, then one has that 
		$
		\Rees_R(\mathcal{I}) \otimes_A \kk \; \simeq \; \Rees_{ R/\nnn R}(I),
		$
		and so one obtains the isomorphism
		$
		{\left[\Rees_R(\mathcal{I}) \otimes_A\kk\right]}_0 \,\simeq\, \kk[\mathbf{\overline{g}}]
		$
		of graded $\kk$-algebras. 
		
		From the Generic Freeness Lemma (see, e.g., \cite[Theorem 14.4]{EISEN_COMM}) and \autoref{rem:zariski_topology_affine}, there exists a dense open subset $W \subset \kk^m$ such that, if $(\alpha_1,\ldots,\alpha_m) \in W$, then 
		$
		 e\left({\left[\Rees_R(\mathcal{I}) \otimes_A\kk\right]}_0\right) = e\left({\left[\Rees_R(\mathcal{I}) \otimes_A\KK\right]}_0\right).
		$
		Therefore, the result follows by taking $\mathcal{Q} = V \cap W$.
	\end{proof}
\end{theorem}

\section{Applications}
\label{sec:applications}

In this section, one provides some applications of the previous developments.

\subsection{Perfect ideals of codimension two}
\label{perf_heiht_two}

Here one deals with the case of a rational map $\FF:\PP_{\kk}^{r} \dashrightarrow \PP_{\kk}^{r}$ with a perfect base ideal of height $2$, where $\kk$ is a field of characteristic zero.

The main idea is that one can compute the degree of the rational map (\cite[Corollary 3.2]{MULT_SAT_PERF_HT_2}) when the condition $G_{r+1}$ is satisfied, then a suitable application of \autoref{THM_REDUCTION_REES} gives an upper bound for all the rational maps that satisfy the weaker condition $F_0$.

Below \autoref{NOTATION_GENERAL1} is adapted to the particular case of  a perfect ideal of height $2$.

\begin{notation}
	\label{NOTATION_PERF_HT_2}
	Let $\kk$ be a field of characteristic zero. 
	Let $1 \le \mu_1 \le \mu_2 \le \cdots \le \mu_r $ be integers with $\mu_1+\mu_2+\cdots+\mu_r=d$.
	For $1 \le i \le r+1$ and $1 \le j \le r$,  let 
	$
	\mathbf{z}_{i,j}=\{z_{i,j,1}, z_{i,j,2}, \ldots, z_{i,j,m_j}\}
	$	
	denote a set of variables over $\kk$, of cardinality $m_j:=\binom{\mu_j+r}{r}$ -- the number of coefficients of a polynomial of degree $\mu_j$ in $r+1$ variables.
	
	Let $\mathbf{z}$ be the set of mutually independent variables $\mathbf{z}= \bigcup_{i,j} \mathbf{z}_{i,j}$,  $A$ be the polynomial ring $A=\kk[\mathbf{z}]$, and $R$ be the polynomial ring $R=A[x_0,\ldots,x_r]$.
	Let $\MM$ be the $(r+1)\times r$ matrix with entries in $R$ given by 
	$$
	\MM = \left( \begin{array}{cccc}
		p_{1,1} & p_{1,2} & \cdots & p_{1,r} \\
		p_{2,1} & p_{2,2} & \cdots & p_{2,r}\\
		\vdots & \vdots & & \vdots\\
		p_{r+1,1} & p_{r+1,2} & \cdots & p_{r+1,r}\\
	\end{array}
	\right)	
	$$	
	where each polynomial $p_{i,j} \in R$ is given by 
	$$
	p_{i,j} = z_{i,j,1} x_0^{\mu_j} + z_{i,j,2} x_0^{\mu_j-1}x_1 + \cdots + z_{i,j,m_j-1}x_{r-1}x_r^{\mu_j-1} + z_{i,j,m_j}x_r^{\mu_j}.
	$$
	Fix a (rational) maximal ideal  $\nnn: =  (z_{i,j,k} - \alpha_{i,j,k})\subset A$ of $A$, with $\alpha_{i,j,k} \in \kk$.
		
	Set $\KK:=\Quot(A)$, and denote $\mathbb{T}: =R\otimes_A\KK= \KK[x_0,\ldots,x_r]$ and $R/\nnn R \simeq \kk[x_0,\ldots,x_r]$.
	
	Let $\mathbb{M}$ and $M$ denote respectively the matrix $\MM$ viewed as a matrix  with entries over $\mathbb{T}$ and $R/\nnn R$.
	Let $\{g_0,g_1,\ldots,g_r\} \subset R$ be the ordered signed maximal minors of the matrix $\MM$.
	Then, the ordered signed maximal minors of $\mathbb{M}$ and $M$ are given by $\{G_0,G_1,\ldots,G_r\} \subset \mathbb{T}$	and $\{\overline{g_0},\overline{g_1},\ldots,\overline{g_r}\} \subset R/\nnn R$, respectively, where $G_i  = g_i \otimes_{R} \mathbb{T}$ and $\overline{g_i}= g_i \otimes_{R} \left(R/\nnn R\right)$.
	
	Let $\GG:\PP_A^r \dashrightarrow \PP_A^r$, $\bGG:\PP_\KK^r \dashrightarrow \PP_\KK^r$ and $\bgg:\PP_\kk^r \dashrightarrow \PP_\kk^r$	be the rational maps given by the representatives $(g_0:\cdots:g_r)$,  $(G_0:\cdots:G_r)$ and $(\overline{g_0}:\cdots:\overline{g_r})$, respectively.
\end{notation}

\begin{lemma}
		\label{lem_properties_perf_ht_2}
		The following statements hold:
		\begin{enumerate}[\rm (i)]
			\item The ideal $I_r(\mathbb{M})$ is perfect of height two and satisfies the condition $G_{r+1}.$
			\item The rational map $\bGG:\PP_\KK^r \dashrightarrow \PP_\KK^r$ is generically finite.
		\end{enumerate}
		\begin{proof}
			Let $\mathbb{I}=I_r(\mathbb{M})$.
			
			(i) The claim that $\mathbb{I}$ is perfect of height two is clear from the Hilbert-Burch Theorem (see, e.g., \cite[Theorem 20.15]{EISEN_COMM}).
			
			From \autoref{prop:ht_gen_det_ideals}, $\HT(I_i(\mathbb{M})) \ge\HT(I_i(\MM))\ge r+2-i$ for $1\le i \le r$.
			Since the $G_{r+1}$ condition on $\mathbb{I}$ (see \autoref{Fitting_conditions}) is equivalent to 
			$
			\HT(I_{r+1-i}(\mathbb{M}))=\HT(\Fitt_{i}(\mathbb{I})) > i
			$
			for $1\le i \le r$, one is through.

			(ii) Note that $\mathbb{I}$ is generated in fixed degree $d=\mu_1+\cdots +\mu_r$. Thus, the image of the rational map $\mathbb{G}$ is $\Proj(\KK[[\mathbb{I}]_d])\subset \PP_\KK^r$.
			Then the generic finiteness of $\bGG$ is equivalent to having $\dim (\KK[[\mathbb{I}]_d])=r+1$.
			On the other hand, $\mathfrak{F}_{\mathbb{T}} (\mathbb{I})\simeq \KK[[\mathbb{I}]_d]$.
			Now, as is well-known, $\dim (\mathfrak{F}_{\mathbb{T}} (\mathbb{I}))= \dim \left(\mathfrak{F}_{\mathbb{T}_{\mathfrak{M}}} (\mathbb{I}_{\mathfrak{M}})\right)
			=\ell(\mathbb{I}_{\mathfrak{M}})$, where $\mathfrak{M}=(x_0,\ldots,x_r)\mathbb{T}$ and $\ell$ stands for the analytic spread.
			
			The ideal $\mathbb{I}_{\mathfrak{M}}\subset  \mathbb{T}_{\mathfrak{M}}$ being perfect of height $2$ is a strongly Cohen--Macaulay ideal (\cite[Theorem 0.2 and Proposition 0.3]{Huneke}), hence in particular satisfies the sliding-depth property (\cite[Definitions 1.2 and 1.3]{Ulrich_Vasc_Eq_Rees_Lin_Present}). 
			
			 Since $\mu({\mathbb{I}_\mathfrak{M}})=\dim(\mathbb{T}_{\mathfrak{M}})$, then part (i) and \cite[Theorem 6.1]{HSV_Approx_Complexes_II} (also \cite[Theorem 9.1]{HSV_TRENTO_SCHOOL}) imply that the analytic spread of $\mathbb{I}$ is equal to $\ell(\mathbb{I}_{\mathfrak{M}})=r+1$, as wished.
		\end{proof}
\end{lemma}

The main result of this subsection is a straightforward application of the previous developments.

\begin{theorem}\label{specializing_degree_HB}
	Let $\kk$ be a field of characteristic zero and let  $B = \kk[x_0,\ldots, x_r]$ denote a polynomial ring over $\kk$.
	Let $I\subset B$ be a perfect ideal of height two minimally generated by $r+1$ forms $\{f_0,f_1,\ldots,f_r\}$ of the same degree $d$ and with a Hilbert-Burch resolution of the form
	$$
	0 \rightarrow \bigoplus_{i=1}^rB(-d-\mu_i) \xrightarrow{\varphi} {B(-d)}^{r+1} \rightarrow I\rightarrow 0.			
	$$
	Consider the rational map $\FF:\PP_{\kk}^{r} \dashrightarrow \PP_{\kk}^{r}$ given by 
	$$
	\left(x_0:\cdots:x_r\right) \mapsto 	\big(f_0(x_0,\ldots,x_r):\cdots:f_r(x_0,\ldots,x_r)\big).	
	$$
	When $\FF$ is generically finite and $I$ satisfies the property $F_0$, one has 
	$
	\deg(\FF) \le \mu_1\mu_2\cdots \mu_r.
	$
	In addition, if $I$ satisfies the condition $G_{r+1}$ then  
	$
	\deg(\FF) = \mu_1\mu_2\cdots \mu_r.
	$
	\begin{proof}
		Let the  $\alpha_{i,j,k}$'s introduced in \autoref{NOTATION_PERF_HT_2} stand for the coefficients of the polynomials in the entries of the presentation matrix $\varphi$.
		Then, under \autoref{NOTATION_PERF_HT_2}, there is a natural isomorphism 
		$\Phi: (A/\nnn)[x_0,\ldots,x_r] \xrightarrow{\simeq} B=\kk[x_0,\ldots,x_r]$
		which, when applied to the entries of the matrix $M$, yields the respective entries of the matrix $\varphi$. 
		Thus, it is equivalent to consider the rational map $\bgg:\PP_\kk^r \dashrightarrow \PP_\kk^r$ determined by the representative $(\overline{g_0}:\cdots:\overline{g_r})$ where $\Phi(\overline{g_i})=f_i$.

 		Since $I_r(\mathbb{M})$ satisfies the $G_{r+1}$ condition (\autoref{lem_properties_perf_ht_2}(i)), \cite[Corollary 3.2]{MULT_SAT_PERF_HT_2} implies that $\deg(\bGG)=\mu_1\mu_2\cdots\mu_r$.

 		Since $\GG$ is generically finite by \autoref{lem_properties_perf_ht_2}(ii) and \autoref{rem:degree_rat_map_over_KK}, its image is the whole of $\PP_A^r$, the latter obviously being a normal scheme.
 		In addition, since $I$ satisfies $F_0$, the conditions of \autoref{THM_REDUCTION_REES}(i)(iv) are satisfied, hence
 		$
 		\deg(\FF) =\deg(\bgg) \le \deg(\GG)=\deg(\bGG)= \mu_1\mu_2\cdots\mu_r.
 		$
 		
 		When $I$ satisfies $G_{r+1}$, then the equality $\deg(\FF)=\mu_1\mu_2\cdots\mu_r$ follows directly from \cite[Corollary 3.2]{MULT_SAT_PERF_HT_2}.
	\end{proof}
\end{theorem}

A particular satisfying case is when $\FF$ is a plane rational map. 
In this case $F_0$ is not a constraint at all, and one recovers the result of \cite[Proposition 5.2]{MULTPROJ}.

\begin{corollary}
	Let $\FF:\PP_\kk^2 \dashrightarrow \PP_\kk^2$ be a dominant rational map defined by a perfect base ideal $I$ of height two. Then,
	$
	\deg(\FF) \le \mu_1\mu_2
	$
	and an equality is attained if $I$ is locally a complete intersection at its minimal primes. 
	\begin{proof}
		In this case property $F_0$ comes for free because  $\HT(I_1(\varphi))\ge \HT(I_2(\varphi))=2$ is always the case.
		Also, here l.c.i. at its minimal primes is equivalent to $G_3$.
	\end{proof}
\end{corollary}

Finally, one shows a simple family of plane rational maps where the degree of the map decreases arbitrarily under specialization.

\begin{example}\rm
	Let $m\ge 1$ be an integer.
	Let $A=\kk[a]$ be a polynomial ring over $\kk$.
	Let $R=A[x,y,z]$ be a standard graded polynomial ring over $A$ and consider the following homogeneous matrix 
	$$
	\arraycolsep=.4cm
	\mathcal{M} = \left(
		\begin{array}{ll}
			x & zy^{m-1}\\
			-y & zx^{m-1} + y^m\\
			az & zx^{m-1}
		\end{array}
	\right)
	$$
	with entries in $R$.
	For any $\beta \in \kk$, let $\nnn_\beta:=(a-\beta) \subset A$.
	Let $\I=(g_0,g_1,g_2):=I_2(\mathcal{M}) \subset R$ and $I_\beta:=\left(\pi_\beta(g_0),\pi_\beta(g_1),\pi_\beta(g_2)\right) \subset R/\nnn_\beta R$ be the specialization of $\I$ via the natural map $\pi_\beta:R \twoheadrightarrow R/\nnn_\beta R$.
	
	Let $\GG:\PP_A^2 \dashrightarrow \PP_A^2$ and $\bgg_\beta : \PP_{A/\nnn_\beta}^2 \dashrightarrow \PP_{A/\nnn_\beta}^2$ be rational maps with representatives $(g_0:g_1:g_2)$ and $\left(\pi_\beta(g_0):\pi_\beta(g_1):\pi_\beta(g_2)\right)$, respectively. 
	
	When $\beta=0$, from \cite[Proposition 2.3]{Hassanzadeh_Simis_Cremona_Sat} one has that $\bgg_0 : \PP_{A/\nnn_0}^2 \dashrightarrow \PP_{A/\nnn_0}^2$ is a de Jonqui\`eres map, which is birational.
	On the other hand, if $\beta \neq 0$, then $I_\beta$ satisfies the condition $G_3$ and so one has $\deg\left(\bgg_\beta\right)=m$.
	
	Therefore, it follows that 
	$$
	\deg\left(\bgg_\beta\right) = \begin{cases}
		1 \qquad\,\, \text{ if } \beta=0, 
		\\
		m \qquad \text{ if } \beta\neq 0.
	\end{cases}
	$$
	Also, note that $\deg(\GG)=m$.
	In concordance with \autoref{THM_REDUCTION_REES}, by setting $\kk = \mathbb{Q}$, one can check with \texttt{Macaulay2} \cite{MACAULAY2} that $\dim\big(\gr_{\I}(R) \otimes_{A} A/(a)\big) = 4$ and that for e.g. $\beta = 1$ one has $\dim\big(\gr_{\I}(R) \otimes_{A} A/(a-1)\big) = 3$.
\end{example}

\subsection{Gorenstein ideals of codimension three}

In this subsection, one deals with the case of a rational map $\FF:\PP_{\kk}^{r} \dashrightarrow \PP_{\kk}^{r}$ with a Gorenstein base ideal of height $3$, where $\kk$ is a field of characteristic zero.
The argument is pretty much the same as in the previous part \autoref{perf_heiht_two}, by
drawing on the formulas in \cite[Corollary C]{MULT_GOR_HT_3}.

\begin{notation}
	\label{NOTATION_GOR_HT_3}
	Let $\kk$ be a field of characteristic zero and $D \ge 1$ be an integer.
	For $1 \le i < j \le r+1$,  let 
	$
	\mathbf{z}_{i,j}=\{z_{i,j,1}, z_{i,j,2}, \ldots, z_{i,j,m} \}
	$	
	denote a set of variables over $\kk$, of cardinality $m:=\binom{D+r}{r}$.
	Let $\mathbf{z}= \bigcup_{i,j} \mathbf{z}_{i,j}$,  $A=\kk[\mathbf{z}]$, and  $R=A[x_0,\ldots,x_r]$.
	Let $\MM$ be the alternating $(r+1)\times (r+1)$ matrix with entries in $R$ given by 
	$$
	\MM = \left( \begin{array}{ccccc}
	0 & p_{1,2} & \cdots & p_{1,r} & p_{1,r+1} \\
	-p_{1,2} & 0 & \cdots & p_{2,r}& p_{2,r+1}\\
	\vdots & \vdots & & \vdots &\vdots\\
	-p_{1,r} & -p_{2,r} & \cdots & 0 & p_{r,r+1}\\
	-p_{1,r+1} & -p_{2,r+1} & \cdots & -p_{r,r+1} & 0\\
	\end{array}
	\right)	
	$$	
	where each polynomial $p_{i,j} \in R$ is given by 
	$$
	p_{i,j} = z_{i,j,1} x_0^{D} + z_{i,j,2} x_0^{D-1}x_1 + \cdots + z_{i,j,m-1}x_{r-1}x_r^{D-1} + z_{i,j,m}x_r^{D}.
	$$
	Fix a (rational) maximal ideal  $\nnn: =  (z_{i,j,k} - \alpha_{i,j,k})\subset A$ of $A$, with $\alpha_{i,j,k} \in \kk$.
	
	Set $\KK:=\Quot(A)$, and denote $\mathbb{T}: =R\otimes_A\KK= \KK[x_0,\ldots,x_r]$ and $R/\nnn R \simeq \kk[x_0,\ldots,x_r]$.
	
	Let $\mathbb{M}$ and $M$ denote respectively the matrix $\MM$ viewed as a matrix  with entries over $\mathbb{T}$ and $R/\nnn R$.
	Let $\{g_0,g_1,\ldots,g_r\} \subset R$ be the ordered signed submaximal Pfaffians of the matrix $\MM$.
	Then, the ordered signed submaximal Pfaffians of $\mathbb{M}$ and $M$ are given by $\{G_0,G_1,\ldots,G_r\} \subset \mathbb{T}$	and $\{\overline{g_0},\overline{g_1},\ldots,\overline{g_r}\} \subset R/\nnn R$, respectively, where $G_i  = g_i \otimes_{R} \mathbb{T}$ and $\overline{g_i}= g_i \otimes_{R} \left(R/\nnn R\right)$.
	
	Let $\GG:\PP_A^r \dashrightarrow \PP_A^r$, $\bGG:\PP_\KK^r \dashrightarrow \PP_\KK^r$ and $\bgg:\PP_\kk^r \dashrightarrow \PP_\kk^r$	be the rational maps given by the representatives $(g_0:\cdots:g_r)$,  $(G_0:\cdots:G_r)$ and $(\overline{g_0}:\cdots:\overline{g_r})$, respectively. 
\end{notation}

\begin{lemma}
	\label{lem_properties_Gor_ht_3}
	The following statements hold:
	\begin{enumerate}[\rm (i)]
		\item The ideal ${\rm Pf}_r(\mathbb{M})$ is Gorenstein of height three and satisfies the condition $G_{r+1}.$
		\item The rational map $\bGG:\PP_\KK^r \dashrightarrow \PP_\KK^r$ is generically finite.
	\end{enumerate}
\end{lemma}
\begin{proof}
	(i) The first claim follows from the Buchsbaum-Eisenbud structure theorem \cite{BuchsbaumEisenbud}.
	The condition $G_{r+1}$ is proved to be satisfied by an argument similar to \autoref{lem_properties_perf_ht_2}(i) (see \cite[Lemma 2.12]{MULT_GOR_HT_3}).
	
	(ii) Essentially the same proof of \autoref{lem_properties_perf_ht_2}(ii) carries over for this case (see \cite[Theorem 0.2 and Proposition 0.3]{Huneke} and \cite[Theorem 6.1]{HSV_Approx_Complexes_II}).
\end{proof}

\begin{theorem}
	\label{thm:Gor_ht_3}
	Let $\kk$ be a field of characteristic zero, $B$ be the polynomial ring $B = \kk[x_0,\ldots, x_r]$ and $I \subset B$ be a homogeneous ideal in $B$.
	Suppose that the following conditions are satisfied:
	\begin{enumerate}[\rm (i)]
		\item $I$ is a height $3$ Gorenstein ideal.
		\item $I$ is minimally generated by $r+1$ forms $\{f_0,f_1,\ldots,f_r\}$ of the same degree.
		\item Every non-zero entry of an alternating minimal presentation matrix of $I$  has degree $D\ge 1$.
	\end{enumerate}
	Consider the rational map $\FF:\PP_{\kk}^{r} \dashrightarrow \PP_{\kk}^{r}$ given by 
	$$
	\left(x_0:\cdots:x_r\right) \mapsto 	\big(f_0(x_0,\ldots,x_r):\cdots:f_r(x_0,\ldots,x_r)\big).	
	$$
	When $\FF$ is generically finite and $I$ satisfies the property $F_0$, one has 
	$
	\deg(\FF) \le D^r.
	$
	In addition, if $I$ satisfies the condition $G_{r+1}$ then  
	$
	\deg(\FF) = D^r.
	$
\end{theorem}
\begin{proof}
	The proof follows similarly to \autoref{specializing_degree_HB}.
	But now one uses \autoref{NOTATION_GOR_HT_3}, \autoref{lem_properties_Gor_ht_3} and \cite[Corollary C]{MULT_GOR_HT_3} instead of \autoref{NOTATION_PERF_HT_2}, \autoref{lem_properties_perf_ht_2} and \cite[Corollary 3.2]{MULT_SAT_PERF_HT_2}, respectively.
\end{proof}

\subsection{$j$-multiplicity}

Here one shows that the $j$-multiplicity of an ideal does not change under a general specialization of the coefficients.
The $j$-multiplicity of an ideal was introduced in \cite{ACHILLES_MANARESI_J_MULT},  as a generalization of the Hilbert-Samuel multiplicity for the non-primary case.
It has been largely applied by several authors in its recent history, including in intersection theory (see, e.g., \cite{FLENNER_O_CARROLL_VOGEL}).

Here one uses the notation of \autoref{NOTA_SAT_SPECIAL_FIB}.
Let $\mm$ be the maximal irrelevant ideal $\mm=(x_0,\ldots,x_r) \subset R/\nnn R$.
 The $j$-multiplicity of the ideal $I \subset R/\nnn R$ is given by
$$
j(I) := r!\lim\limits_{n\rightarrow\infty} \frac{\dim_{\kk}\left(\HL^0\left(I^n/I^{n+1}\right)\right)}{n^r}.
$$
The $j$-multiplicity of $\mathbb{I} \subset \mathbb{T}$ is defined in the same way.

\begin{corollary}
	\label{cor:j_mult}
	Under {\rm \autoref{NOTA_SAT_SPECIAL_FIB}}, suppose that $\mathbb{I}$ has maximal analytic spread $\ell(\mathbb{I})=r+1$.
	Then, there is a dense open subset $\mathcal{V} \subset \kk^m$ such that, if $\nnn=(z_1-\alpha_1, \ldots, z_m-\alpha_m)$ with $(\alpha_1, \ldots, \alpha_m)\in \mathcal{V}$, one has
	$$
	j(I) = j(\mathbb{I}).
	$$
\end{corollary}
\begin{proof}
	Let $d = \deg(\overline{g_i}) = \deg(G_i)$.
	From \cite[Theorem 5.3]{KPU_blowup_fibers} one obtains 
	$$
	j(I)=d\cdot \deg(\bgg)\cdot \deg_{\PP_\kk^s}(Y)  \quad \text{ and } \quad j(\mathbb{I})= d \cdot  \deg(\bGG)\cdot \deg_{\PP_\KK^s}(\mathbb{Y}).
	$$
	So, the result follows from \autoref{specializing_mul_fiber}.
\end{proof}

\section*{Acknowledgments}

We are grateful to the referee for valuable comments and suggestions for the improvement of this work.
The first named author was funded by the European Union's Horizon 2020 research and innovation programme under the Marie Sk\l{}odowska-Curie grant agreement No. 675789. 
The second named author was funded by a Senior Professor Position at the Politecnico di Torino and a Brazilian CNPq grant (302298/2014-2). Both authors thank the Department of Mathematics of the Politecnico for providing an appropriate environment for discussions on this work.

\bibliographystyle{elsarticle-num} 
\input{Degree_rational_maps_and_specialization.bbl}

\end{document}